\documentclass[10pt]{article}
 
\usepackage{coursE}
\usepackage{color}

\numberwithin{equation}{section}

\title{Littlewood-Paley functionals on graphs}

\author{{\Large Joseph {\sc Feneuil} \footnote{The author is supported by the ANR project ``Harmonic Analysis at its Boundaries'',   ANR-12-BS01-0013-03.}} 
\\
Institut Fourier, UMR 5582\\
100 rue des Math\'ematiques, BP 74, F-38402 
Saint-Martin d'H\`eres, France\\
joseph.feneuil@ujf-grenoble.fr}
\date{\today}

\begin{document}

\maketitle

\begin{abstract}
Let $\Gamma$ be a graph equipped with a Markov operator $P$. 
We introduce discrete fractional Littlewood-Paley square functionals and prove their $L^p$-boundedness under various geometric assumptions on the graph $\Gamma$.
\end{abstract}

\tableofcontents

\pagebreak

Throughout the paper, we use the following notations. If $E$ is a nonempty set and $A$ and $B$ are some quantities depending on $x\in E$, the notation $A(x) \lesssim B(x)$ means that there exists $C$ such that $A(x) \leq C \,  B(x)$ for all $x\in E$, while $A(x) \simeq B(x)$ means that $A(x) \lesssim B(x)$ and $B(x) \lesssim A(x)$. \par
\noindent If $E$ and $F$ are Banach spaces and $T:E\rightarrow F$ is a bounded linear operator, $\left\Vert T\right\Vert_{E\rightarrow F}$ stands for the operator norm of $T$. When $E=L^p$ and $F=L^q$ for $1\leq p,q\leq +\infty$, $\left\Vert T\right\Vert_{L^p\rightarrow L^q}$ will also be denoted by $\left\Vert T\right\Vert_{p,q}$.\par

\section{Introduction}

This paper is devoted to the $L^p$-boundedness of Littlewood-Paley type square functionals on graphs. The prototype of these functionals is the $g$-function in the Euclidean space, defined in the following way. If $f$ is, say, in ${\mathcal D}(\R^n)$ and $u(x,t)$ denotes ``the'' harmonic extension of $f$, that is $u(x,t)=P_t\ast f(x)$ for all $t>0$ and all $x\in \R^n$, where $P_t$ stands for the Poisson kernel, define
\[
g_1f(x):=\left(\int_0^{+\infty} \left(\left\vert \frac{\partial u}{\partial t}(x,t)\right\vert^2+\sum_{1\leq i\leq n} \left\vert \frac{\partial u}{\partial x_i}(x,t)\right\vert^2\right) \frac{dt}t\right)^{1/2}.
\]
It is a well-known fact (\cite[Chapter 4, Theorem 1]{Steinsingular}) that, for all $p\in (1,+\infty)$,
\begin{equation} \label{Fen1equivalenceg}
\left\Vert   g_1f   \right\Vert_{L^p(\R^n)}\sim \left\Vert f\right\Vert_{L^p(\R^n)}.
\end{equation}
This result was extended in various directions, and we only recall some of them. In the Euclidean framework, the harmonic extension can be replaced by $e^{-tL}$, where $L$ is a second order uniformly elliptic operator in divergence form. In this case, the range of $p$ in \eqref{Fen1equivalenceg} is related to the $L^p$ boundedness of $e^{-tL}$ or $t\nabla e^{-tL}$ (see \cite[Chapter 7]{Auscher2007}). \par
\noindent If, in the functional $g$, one is only interested in the ``horizontal'' part, {\it i.e. } the derivative with respect to $t$, then the $L^p$ boundedness of the corresponding Littlewood-Paley functional holds in the much more general context of measured spaces endowed with appropriate Markov semigroups (\cite[Corollaries 1 and 2]{Stein1}). Notice also that similar results can be proved when the derivative $\frac{\partial }{\partial t}$ is replaced by a ``fractional'' derivative (\cite{CRW}).\par
\noindent Littlewood-Paley functionals were also considered in the context of complete Riemannian manifolds. Let $M$ be a complete Riemannian manifold, $\nabla$ be the Riemannian gradient and $\Delta$ the Laplace-Beltrami operator. Consider the ``vertical'' functionals
\[
Gf(x):=\left(\int_0^{+\infty} \left\vert t\nabla e^{-t\sqrt{\Delta}}f(x)\right\vert^2\frac{dt}t\right)^{1/2}
\]
and
\[
Hf(x):=\left(\int_0^{+\infty} \left\vert \sqrt{t}\nabla e^{-t\Delta}f(x)\right\vert^2\frac{dt}t\right)^{1/2}.
\]
Several $L^p$-boundedness results for $G$ and $H$ are known. Let us recall here that, when $1<p\leq 2$, $G$ and $H$ are $L^p(M)$-bounded when $M$ is an {\it arbitrary} complete Riemannian manifold (\cite[Theorem 1.2]{CDL}), while the $L^p(M)$-boundedness of $G$ and $H$ for $p>2$ holds under much stronger assumptions, expressed in terms of the domination of the gradient of the semigroup by the semigroup applied to the gradient (\cite[Proposition 3.1]{CDgeq2}). \par
\noindent Littlewood-Paley functionals on graphs were also considered. In \cite{Dungey2}, if $\Delta$ is a Laplace operator on a graph $\Gamma$, a ``vertical'' Littlewood-Paley functional, involving the   (continuous-time)    semigroup generated by $\Delta$, is proved to be $L^p(\Gamma)$-bounded for all $1<p\leq 2$ under very weak assumptions on $\Gamma$. In \cite{BRuss}, ``discrete time'' Littlewood-Paley functionals are proved to be $L^p(\Gamma)$-bounded under geometric assumptions on $\Gamma$ (about the volume growth of balls, or $L^2$ Poincar\'e inequalities), while similar results are obtained for weighted $L^p$-norms in \cite{BMartell}. Note also that the $L^p$-boundedness of discrete time Littlewood-Paley functionals in abstract settings was recently established in \cite{al}. \par

\noindent The present paper is devoted to the proof of the $L^p$-boundedness   on graphs    of some discrete time fractional Littlewood-Paley horizontal or vertical functionals. Before stating our results, let us present the graphs under consideration.

\subsection{Presentation of the discrete framework}

\subsubsection{General setting} \label{Fen1DefGraphs}

Let $\Gamma$ be an infinite set and $\mu_{xy} = \mu_{yx} \geq 0$ a symmetric weight on $\Gamma \times \Gamma$. 
The couple $(\Gamma, \mu)$ induces a (weighted unoriented) graph structure if we define the set of edges by
\[E = \{ (x,y) \in \Gamma \times \Gamma, \, \mu_{xy} >0 \}.\]
We call then $x$ and $y$ neighbors (or $x\sim y$) if $(x,y) \in E$.\par
\noindent We will assume that the graph is connected and locally uniformly finite. 
A graph is connected if for all $x,y \in \Gamma$, there exists a path $x = x_0,x_1, \dots,x_N = y$ such that for all $1\leq i\leq N$, $x_{i-1}\sim x_i$ (the length of such path is then $N$).
A graph is said to be locally uniformly finite if there exists $M_0 \in \N$ such that for all $x\in \Gamma$, $\# \{y\in \Gamma, \, y\sim x\} \leq M_0$ (i.e. the number of neighbors of a   vertex    is uniformly bounded).\par
\noindent The graph is endowed with its natural metric $d$, which is the shortest length of a path joining two points. 
  For all $x\in \Gamma$ and all $r>0$, the ball of center $x$ and radius $r$ is defined as $B(x,r) = \{y\in \Gamma, \, d(x,y) <r\}$. 
In the opposite way, the radius of a ball $B$ is the only integer $r$ such that $B=B(x_B,r)$ (with $x_B$ the center of $B$). 
Therefore, for all balls   $B=B(x,r)$ and all $\lambda>0$, we set $\lambda B:=B(x,\lambda r)$ and define    $C_j(B) = 2^{j+1}B \backslash 2^jB$ for all $j\geq 2$ and $C_1(B) = 4B$.\par
\noindent We define the weight $m(x)$ of a vertex $x \in \Gamma$ by $m(x) = \sum_{x\sim y} \mu_{xy}$.
More generally, the volume of a subset $E \subset \Gamma$ is defined as $m(E) := \sum_{x\in E} m(x)$. We use the notation $V(x,r)$ for the volume of the ball $B(x,r)$, and in the same way, $V(B)$ represents the volume of a ball $B$. \par
\noindent We define now the $L^p(\Gamma)$ spaces. For all $1\leq p < +\infty$, we say that a function $f$ on $\Gamma$ belongs to $L^p(\Gamma,m)$ (or $L^p(\Gamma)$) if
\[\|f\|_p := \left( \sum_{x\in \Gamma} |f(x)|^p m(x) \right)^{\frac{1}{p}} < +\infty,\]
while $L^\infty(\Gamma)$ is the set of functions satisfying 
\[\|f\|_{\infty} : = \sup_{x\in\Gamma} |f(x)| <+\infty.\]
Let us define for all $x,y\in \Gamma$ the discrete-time reversible Markov kernel $p$ associated to the measure $m$ by $p(x,y) = \frac{\mu_{xy}}{m(x)m(y)}$.
The discrete kernel $p_l(x,y)$ is then defined recursively for all $l\geq 0$ by
\begin{equation}
\left\{ 
\begin{array}{l}
p_0(x,y) = \frac{\delta(x,y)}{m(y)} \\
p_{l+1}(x,y) = \sum_{z\in \Gamma} p(x,z)p_l(z,y)m(z).
\end{array}
\right.
\end{equation}

\begin{rmq}
Note that this definition of $p_l$ differs from the one of $p_l$ in \cite{Russ}, \cite{BRuss} or \cite{Delmotte1},   because of the $m(y)$ factor.    However,   $p_l$ coincides with $K_l$    in \cite{Dungey}. 
Remark that in the case of the Cayley   graphs    of finitely generated discrete groups, where $m(x)=1$ for all $x$, the definitions   coincide.   
\end{rmq}

Notice that for all $l\geq 1$, we have
\begin{equation} \label{Fen1sum=1}
 \|p_l(x,.)\|_{L^1(\Gamma)} = \sum_{y\in \Gamma} p_l(x,y)m(y) = \sum_{d(x,y) \leq l} p_l(x,y)m(y) = 1 \qquad \forall x\in \Gamma,
\end{equation}
and that the kernel is symmetric:
\begin{equation} \label{Fen1symmetry}
 p_l(x,y) = p_l(y,x) \qquad \forall x,y\in \Gamma.
\end{equation}
For all functions $f$ on $\Gamma$, we define $P$ as the operator with kernel $p$, i.e.
\begin{equation}\label{Fen1defP} Pf(x) = \sum_{y\in \Gamma} p(x,y)f(y)m(y) \qquad \forall x\in \Gamma.\end{equation}
It is easily checked that $P^l$ is the operator with kernel $p_l$.

\begin{rmq}
Even if the definition of $p_l$ is different from \cite{Russ} or \cite{BRuss}, $P^l$ is the same operator in both cases. 
\end{rmq}

Since $p(x,y) \geq 0$ and \eqref{Fen1sum=1} holds, one has, for all $p\in [1,+\infty]$ ,
\begin{equation} \label{Fen1Pcont}
 \|P\|_{p\to p } \leq 1.
\end{equation}
\begin{rmq} \label{Fen1poweri-p}
Let $1<p<+\infty$. Since, for all $l\geq 0$, $\left\Vert P^l\right\Vert_{p\rightarrow p}\leq 1$, the operators $(I-P)^{\beta}$ and $(I+P)^{\beta}$ are $L^p$-bounded for all $\beta>0$ (see \cite{CSC}, p. 423).
\end{rmq}

\noindent We define a nonnegative Laplacian on $\Gamma$ by $\Delta = I-P$. One has then
\begin{equation}
 \begin{split}
 <(I-P)f,f>_{L^2(\Gamma)} & = \sum_{x,y\in \Gamma} p(x,y)(f(x)-f(y))f(x)m(x)m(y) \\
& = \frac{1}{2} \sum_{x,y\in \Gamma} p(x,y)|f(x)-f(y)|^2m(x)m(y),
 \end{split}
\end{equation}
where we use \eqref{Fen1sum=1} for the first equality and \eqref{Fen1symmetry} for the second one. The last calculus proves that the following operator
\[\nabla f(x) = \left( \frac{1}{2} \sum_{y\in \Gamma} p(x,y) |f(y)-f(x)|^2 m(y)\right)^{\frac{1}{2}},\]
called ``length of the gradient'' (and the definition of which is taken from \cite{CoulGrigor}), satisfies
\begin{equation}
 <(I-P)f,f>_{L^2(\Gamma)} = \|\nabla f\|^2_{L^2(\Gamma)}.
\end{equation}

\subsubsection{Geometric assumptions and estimates for the Markov operator}
Under suitable geometric assumptions on $\Gamma$, the iterates of $P$ satisfy various $L^p-L^q$ estimates, which we now review.\par
\noindent Our first assumption is:

\begin{defi}
 A graph $(\Gamma, \mu)$ satisfies \eqref{Fen1LB} if there exists $\epsilon >0$ such that
\begin{equation} \label{Fen1LB} \tag{$LB$} \mu_{xx} \geq \epsilon m(x) \qquad \forall x\in \Gamma. \end{equation}
\end{defi}

\begin{rmq}
Let us state a stronger assumption than \eqref{Fen1LB}: there exists  $\epsilon >0$ such that for all $x\in \Gamma$, $x\sim x$ and
\begin{equation} \label{Fen1LB2} \tag{$LB_2$}
\mu_{xy} \geq \epsilon m(x) \qquad \forall x\sim y.
\end{equation}
Even if \eqref{Fen1LB2} plays a crucial role in some parabolic regularity estimates on graphs (\cite{Delmotte1}), it will play no role in our results.
\end{rmq}

The second assumption is the following one: 

\begin{defi}[Doubling property]
 The weighted graph $(\Gamma,\mu)$ satisfies the doubling property if there exists $C>0$ such that
\begin{equation} \label{Fen1DV} \tag{$DV$}
 V(x,2r) \leq C V(x,r) \qquad \forall x\in \Gamma, \, \forall r>0.
\end{equation}
\end{defi}
Recall that, under the assumption \eqref{Fen1DV}, there exists $d>0$ such that
\begin{equation} \label{Fen1PDV}
V(\theta x,r) \lesssim \theta^d  V(x,r) \qquad \forall r>0, \, x\in \Gamma, \, \theta\geq 1.
\end{equation}
In the sequel, a local version of \eqref{Fen1DV} will also be needed:
\begin{defi}
Say that $(\Gamma,\mu)$ satisfies \eqref{Fen1LDV} if there exists $c >0$ such that
\begin{equation} \label{Fen1LDV} \tag{LDV}
  V(x,2)\leq cm(x)    \qquad \forall x\in \Gamma.
\end{equation}
\end{defi}

Let us also state the Poincar\'e inequalities needed in the sequel.
\begin{defi}[Poincaré inequality on balls]
Let $s\in [1,+\infty)$. The weighted graph $(\Gamma,\mu)$ satisfies the Poincar\'e inequality \eqref{Fen1Ps} if there exists $C>0$ such that, for all $x \in \Gamma$, all $r>0$ and all functions on $\Gamma$

\begin{equation} \label{Fen1Ps} \tag{$P_s$}
\frac{1}{V(x,r)} \sum_{y\in B(x,r)} |f(y)-f_B|^s m(y) \leq C \frac{r^s}{V(x,2r)} \sum_{y\in B(x,2r)} |\nabla f(y)|^s m(y),
\end{equation}
where
\begin{equation} \label{Fen1defmean}f_B = \frac{1}{V(B)} \sum_{x\in B} f(x)m(x). \end{equation}
\end{defi}

\begin{rmq}
It is a known fact that ($P_{s_1}$) implies ($P_{s_2}$) if $s_1 \leq s_2$ (cf \cite{HK}).
\end{rmq}

\noindent Let us now introduce some estimates on $p_l$, which will be needed in the statement of our results.
\begin{defi}[On diagonal upper estimate of $p_l$]
We say that $(\Gamma,\mu)$ satisfies \eqref{Fen1DUE} if there exists $C>0$ such that, for all $x \in \Gamma$ and all $l\in \N^*$,
\begin{equation} \label{Fen1DUE} \tag{$DUE$}
 p_l(x,x) \leq \frac{C}{V(x,\sqrt l)} .
\end{equation}
\end{defi}
\begin{defi}
Let $p\in [1,+\infty]$. Say that a weighted graph $(\Gamma,\mu)$ verifies \eqref{Fen1GGp} if
\begin{equation} \tag{$GG_p$} \label{Fen1GGp}
 \|\nabla P^l f\|_{L^p} \leq \frac{C_p}{\sqrt l} \|f\|_{L^p}\  \forall l\in \N^*, \, \forall f\in L^p(\Gamma).
\end{equation}
\end{defi}

\begin{rmq}
 Note that the assumption ($GG_\infty$)   holds    when $\Gamma$ is the Cayley graph of a finitely generated discrete group (as well as assumption $(P_1)$, see \cite{HK}). Indeed, in this case, 
\[
\nabla_x p_l(x,y) \lesssim \left( \frac{1}{lV(x,\sqrt l)V(y,\sqrt l)} \right)^{\frac{1}{2}} \exp\left(-c \frac{d^2(x,y)}{l} \right).
\]
\end{rmq}

\subsection{Main results}

For all $\beta>0$, all functions $f$ on $\Gamma$ and all $x\in \Gamma$, define
\[g_\beta f(x) = \left( \sum_{l\geq 1} l^{2\beta-1} \left| (I-P)^{\beta} P^{l-1} f(x) \right|^2\right)^{\frac{1}{2}}.\]
For all $\beta>-\frac 12$, all functions $f$ on $\Gamma$ and all $x\in \Gamma$, define
 \[\tilde g_\beta f(x) = \left( \sum_{l\geq 1} l^{2\beta} \left| \nabla (I-P)^{\beta} P^{l-1} f(x) \right|^2\right)^{\frac{1}{2}}.\]

  Here is our main result:   
\begin{theo} \label{Fen1maintheo}
Let $(\Gamma,\mu)$ be a weighted graph satisfying \eqref{Fen1DV}, \eqref{Fen1LB} and \eqref{Fen1DUE}. 
Then 
\begin{enumerate}
\item $g_\beta$ is of weak type $(1,1)$, which means that there exists $C>0$ such that, for all $\lambda>0$,
\[
m\left(\left\{x\in \Gamma;\ g_{\beta}f(x)>\lambda\right\}\right)\leq \frac C{\lambda} \left\Vert f\right\Vert_{L^1(\Gamma)},
\]
and of strong type $(p,p)$ for all $1<p<+\infty$, i.e. there exists a constant $C>0$ such that
\[\|g_\beta f\|_{L^p} \leq C \|f\|_{L^p} \qquad \forall f\in L^p(\Gamma) \cap L^2(\Gamma).\]
\item $\tilde g_\beta$ is of weak type $(1,1)$, and of strong type $(p,p)$ for all $1<p\leq 2$. 
Moreover, if $(\Gamma,\mu)$ satisfies $(P_2)$ and ($GG_q$) for some $q>2$, then $\tilde g_\beta$ is of strong type $(p,p)$ for $p\in (2,q)$.
\item For all $1<p<+\infty$,
\[\| f\|_{L^p} \leq C \|g_\beta f\|_{L^p} \qquad \forall f\in L^p(\Gamma) \cap L^2(\Gamma),\]
for all $2 \leq p < +\infty$
\[\| f\|_{L^p} \leq C \|\tilde g_\beta f\|_{L^p} \qquad \forall f\in L^p(\Gamma) \cap L^2(\Gamma)\]
and if $(P_2)$ and ($GG_q$) are true for some $q>2$, then for all $q'<p<2$ (with $\frac{1}{q'} + \frac{1}{q} = 1$),
\[\| f\|_{L^p} \leq C \|\tilde g_\beta f\|_{L^p} \qquad \forall f\in L^p(\Gamma) \cap L^2(\Gamma).\]
\end{enumerate}
\end{theo}

Our second result deals with the $L^p$-boundedness of $\tilde g_0$, under very weak assumptions on $\Gamma$:

\begin{theo} \label{Fen1Maintheo}
Let $(\Gamma,\mu)$ be a graph satisfying \eqref{Fen1LB} and \eqref{Fen1LDV}. Then $\tilde g_0$ is $L^p$-bounded for all $p\in (1,2]$.
\end{theo}

\begin{rmq}
The range $\beta>-\frac{1}{2}$ for the $L^p$-boundedness of $\tilde g_\beta$ is related to the presence of $\nabla$ in $\tilde{g}_{\beta}$.
\end{rmq}

\begin{rmq}
\begin{itemize}
\item[$1.$]
The $L^p$-boundedness of $g_1$ was proved in \cite[Theorem 1.16]{BRuss}. Theorem \ref{Fen1maintheo} extends this fact to a fractional version of $g_1$. Moreover, we prove a similar estimate for the vertical Littlewood-Paley functional $\tilde g_\beta$ and also establish converse inequalities. 
\item[$2.$]
The $L^p$-boundedness of $g_{\beta}$ can be deduced from arguments in \cite{al}. Indeed,   since    $g_1$ is of strong type $(p,p)$ for all $p\in(1,+\infty)$ by \cite[Theorem 1.16]{BRuss}, \cite[Theorem 3.1]{al} yields that $P$ is an $R$-Ritt operator, and the fact that $g_{\beta}$ is of strong type $(p,p)$ for all $p\in (1,+\infty)$ follows from \cite[Theorem 3.3]{al}. However, these arguments do not yield the fact that $g_{\beta}$ is of weak type $(1,1)$. Moreover, they do not provide any information about $\tilde{g}_{\beta}$. 
\end{itemize}
\end{rmq}

Section 2 is devoted to the introduction of the tools used in the sequel. 
In particular, we state various off-diagonal estimates of the Markov kernel, which are proven in the Appendix A. 
Theorems \ref{Fen1maintheo} and \ref{Fen1Maintheo} are respectively proven in Section 3 and 4.

\ms

{\bf Acknowledgements: } the author would like to thank C. Le Merdy for pointing out reference \cite{al} to him.

\section{Preliminary results}

\label{Fen1tools}

\subsection{Estimates on the kernels}
In this paragraph, we gather various estimates on $p_l$ which will be instrumental in our proofs.
The conjunction of \eqref{Fen1LB}, \eqref{Fen1DV} and \eqref{Fen1DUE} provide us with further estimates on $p_l$. First, one has (\cite[Theorem 5.2, Theorem 6.1]{CGZ}): 
\begin{prop}
Let $(\Gamma, \mu)$ be a weighted graph satisfying \eqref{Fen1DV} and \eqref{Fen1LB}. Then, assumption \eqref{Fen1DUE} is equivalent to the off-diagonal upper estimate:
\begin{equation} \label{Fen1UE} \tag{$UE$}
 p_l(x,y) \leq C \left( \frac{1}{V(x,\sqrt l)V(y,\sqrt l)} \right)^{\frac{1}{2}} \exp\left( -c \frac{d^2(x,y)}{l}\right) \qquad \forall x,y\in \Gamma, \ \forall l\in \N^*.
\end{equation}
\end{prop}
\begin{rmq} \label{Fen1pl-1}
An immediate consequence of \eqref{Fen1DV} is that, for all $x,y\in \Gamma$ and $l \in \N^*$,
\[p_{l-1}(x,y) \leq C \left( \frac{1}{V(x,\sqrt l)V(y,\sqrt l)} \right)^{\frac{1}{2}} \exp\left( -c \frac{d^2(x,y)}{l}\right).\]
\end{rmq}
\begin{rmq}
 Assume that $\Gamma$ is a graph satisfying \eqref{Fen1DV}. It is easily checked that assumption \eqref{Fen1UE} is equivalent to
\begin{equation}
 p_l(x,y) \leq \frac{C}{V(y,\sqrt l)} \exp\left( -c \frac{d^2(x,y)}{l}\right)
\end{equation}
or 
\begin{equation}
 p_l(x,y) \leq \frac{C}{V(x,\sqrt l)} \exp\left( -c \frac{d^2(x,y)}{l}\right).
\end{equation}
\end{rmq}
We will now state some ``time regularity'' estimates for higher order differences of $p_l$ (first proved for first order differences by Christ (\cite{christ}) but an easier proof was given by Dungey in \cite{Dungey}). 
\begin{theo} \label{Fen1MTDUE}
Let $(\Gamma,\mu)$ be a weighted graph. Assume that $\Gamma$ satisfies \eqref{Fen1DV}, \eqref{Fen1LB} and \eqref{Fen1DUE}. We define $D(r)$ as the following operator which acts on sequences 
\[(D(r)u)_l = u_l - u_{l+r}.\]
Then, for all $j\geq 0$ there exist two constants $C_j,c_j>0$ such that, for all $l\geq 1$ and all $x,y\in \Gamma$,
\begin{equation} \tag{$TD-UE$} \label{Fen1TDUE} 
  |(D(1)^jp)_l(x,y)| \leq \frac{C_j}{l^j V(x,\sqrt l)} \exp\left( -c_j \frac{d^2(x,y)}{l}\right).
\end{equation}
\end{theo}
  Theorem \ref{Fen1MTDUE} (actually a slightly more general version) will be established in Section \ref{Fen1tools2} in the appendix.    From the previous estimates, we derive the following result, the proof of which will be given in   Section \ref{Fen1tools3} in the appendix.   
\begin{theo} \label{Fen1GTineq}
 Let $(\Gamma,\mu)$ be a weighted graph satisfying \eqref{Fen1DV}, \eqref{Fen1LB} and \eqref{Fen1DUE}. The following Gaffney type inequalities hold: for all $j\in \N$, there exist $c,C>0$ such that for all sets $E,F\subset \Gamma$, all $x_0\in \Gamma$, all $l\in \N^*$ satisfying one of the following conditions
\begin{enumerate}[(i)]
 \item $\sup\left\{ d(x_0,y), \, y\in F \right\} \leq 3 d(E,F)$,
\item $\sup\left\{ d(x_0,y), \, y\in F \right\} \leq \sqrt l$,
 \item $\sup\left\{ d(x_0,x), \, x\in E \right\} \leq 3 d(E,F)$,
\item $\sup\left\{ d(x_0,x), \, x\in E \right\} \leq \sqrt l$,
\end{enumerate}
and all functions $f$ supported   in    $F$, we have, for all $j\in \N$,
\begin{equation} \label{Fen1GT2} \tag{$GT_2$} \begin{split}
\| (I-P)^j P^l f\|_{L^2(E)} & \leq \frac{C}{l^{j}} \frac{1}{V(x_0,\sqrt l)^{\frac{1}{2}}}e^{-c\frac{d(E,F)^2}{l}} \|f\|_{L^1(F)}  \\
\end{split}\end{equation} 
and
\begin{equation} \label{Fen1GGT2} \tag{$GGT_2$} \begin{split} 
\|\nabla (I-P)^j P^l f\|_{L^2(E)} & \leq \frac{C}{l^{j+\frac{1}{2}}} \frac{1}{V(x_0,\sqrt l)^{\frac{1}{2}}} e^{-c\frac{d(E,F)^2}{l}} \|f\|_{L^1(F)} \\
 \|\nabla (I-P)^j P^l f\|_{L^2(E)} & \leq \frac{C}{l^{j+\frac{1}{2}}} e^{-c\frac{d(E,F)^2}{l}} \|f\|_{L^2(F)}.
\end{split} \end{equation} 
\end{theo}

\begin{rmq}
The theorem above will be used for 
\[(E,F) \in \{(B,C_j(B)), \, B \text{ ball }, j\geq 2 \} \cup \{(C_j(B),B), \, B \text{ ball }, j\geq 2 \}.\]
\end{rmq}

\subsection{Results on the Hardy-Littlewood maximal function}

\begin{defi}
 Denote by $\mathcal M$ the Hardy-Littlewood maximal operator
\[\mathcal M f(x) = \sup \frac{1}{V(B)} \sum_{y\in B} |f(y)| m(y)\]
where the supremum is taken over the balls $B$ of $\Gamma$ containing $x$.\par
\noindent In the same way, for $s\geq 1$, $\mathcal M_s$ will denote
\[\mathcal M_s f = \left(\mathcal M |f|^s \right)^{\frac{1}{s}}.\]
\end{defi}
The following observation will turn to be useful: under the assumption \eqref{Fen1UE}, for all $k\geq 1$, all functions $f$ on $\Gamma$ and all $x_0,x\in \Gamma$ with $d(x,x_0)\leq \sqrt{k}$,
\begin{equation}\label{Fen1dominmax}
\left\vert P^kf(x)\right\vert\leq {\mathcal M}f(x_0).
\end{equation}
Indeed,
\[
 \begin{split}
  \left|P^k f(x)\right| & = \left|\sum_{y\in \Gamma} p_k(x,y) f(y)m(y)\right| \\
     & \lesssim \sum_{y\in \Gamma}  \frac{1}{V(x,\sqrt k)} \exp \left(-c \frac{d(x,y)^2}{k}\right) |f(y)|m(y)  \\
& \lesssim \frac{1}{V(x,\sqrt k)} \sum_{d(x,y) < \sqrt k} |f(y)|m(y)  +   \sum_{j\geq 0} \frac{e^{-c2^{2j}}}{V(x,\sqrt k)} \sum_{2^j\sqrt k \leq d(x,y) < 2^{j+1}\sqrt k} 	|f(y)|m(y) \\
& \begin{split} \lesssim  \frac{1}{V(x,\sqrt k)} & \sum_{d(x,y)< \sqrt k} |f(y)|m(y)     \\
& \sum_{j\geq 0} \frac{2^{(j+1)d}e^{-c2^{2j}}}{V(x,2^{j+1}\sqrt k)} \sum_{2^j\sqrt k \leq d(x,y) < 2^{j+1}\sqrt k} 	|f(y)|m(y) \end{split}\\
& \quad \leq \left( 1 + \sum_{j \geq 0} 2^{(j+1)d}e^{-c2^{2j}} \right)	\mathcal M f(x_0) \\
& \lesssim \mathcal M f(x_0),
 \end{split}
\]
where we use for the fifth line the doubling property and the fact that $d^2(x,x_0)\leq k$.  
\begin{prop} \label{Fen1MaxFunct}
 Let $(\Gamma,\mu)$ be a weighted graph satisfying \eqref{Fen1DV}. 
If $(q,q_0,\beta) \in (1,+\infty]^2 \times [0,1)$ satisfy $\frac{1}{q} = \frac{1}{q_0} - \beta$, then $\mathcal M^\beta$ is bounded from $L^{q_0}(\Gamma)$ to $L^q(\Gamma)$.
\end{prop}

\noindent We also recall the Fefferman-Stein inequality.

\begin{theo} \label{Fen1FeffermanStein}
Let $(\Gamma,\mu)$ be a weighted graph satisfying \eqref{Fen1DV} and $s\geq 1$. Then, if $p,q\in (s,+\infty)$, there exists $C_{p,q}>0$ such that for all sequences $(f_n)_{n\in \N}$ of measurable functions defined on $\Gamma$,
\[\left\| \left[ \sum_{n=0}^{+\infty} (\mathcal M_s f_n)^q \right]^{\frac{1}{q}}\right\|_{p} \leq C_{p,q} \left\| \left[ \sum_{n=0}^{+\infty} |f_n|^q \right]^{\frac{1}{q}}\right\|_{p}.\]
\end{theo}

\noindent This result is proven in $\R^d$ in \cite{FS1} and the proof easily extends to spaces of homogeneous type.

\subsection{$L^p$ boundedness for Calder\'on-Zygmund operators}

We will make use of the following theorems about Calder\'on-Zygmund operators ``without kernels'', which can be found in    \cite{BRuss}, Theorem 1.14 and Theorem 1.17. See also    \cite{Auscher2007}, Theorem 1.1 and 1.2.\par
\noindent Before stating these results, recall   (see Theorem \ref{Fen1maintheo})    that a sublinear operator $T$ is of weak type $(p,p)$ ($1\leq p<+\infty$) if there exists $C>0$ such that, for all $\lambda>0$ and all $f\in L^{p}(\Gamma)$,
\[
m\left(\left\{x\in \Gamma;\ \left\vert Tf(x)\right\vert>\lambda\right\}\right)\leq \frac C{\lambda^{p}}\left\Vert f\right\Vert_{L^{p}(\Gamma)}^{p}.
\]
Furthermore, $T$ is said to be of strong type $(p,p)$ if there exists $C>0$ such that, for all $f\in L^p(\Gamma)$,
\[
\left\Vert Tf\right\Vert_{L^p(\Gamma)}\leq C\left\Vert f\right\Vert_{L^p(\Gamma)}.
\]
  
\begin{theo} \label{Fen1interpolation1}
Let $p_0 \in (2,+\infty]$. Assume that $\Gamma$ satisfies the doubling property \eqref{Fen1DV} and let $T$ be a sublinear operator of strong type $(2,2)$ defined on $\Gamma$.
For all balls $B$, let $A_B$ be a linear operator acting on $L^2(\Gamma)$. 
Assume that there exists a constant $C>0$ such that, for all $f\in L^2(\Gamma)$, all $x\in \Gamma$ and all balls $B \ni x$,
\begin{equation} \label{Fen1inter1} \frac{1}{V(B)^{\frac{1}{2}}} \left\| T(I-A_B)f\right\|_{L^2(B)} \leq C \mathcal M_2f(x) \end{equation}
and
\begin{equation} \label{Fen1inter2} \frac{1}{V(B)^{\frac{1}{p_0}}} \left\| TA_Bf\right\|_{L^{p_0}(B)} \leq C  \mathcal M_2|Tf|(x).\end{equation}
Then, for all $p\in (2,p_0)$, $T$ is of strong type $(p,p)$.
\end{theo}

\begin{theo} \label{Fen1interpolation2}

Let $p_0 \in [1,2)$. Assume that $\Gamma$ satisfies the doubling property \eqref{Fen1DV} and let $T$ be a sublinear operator of strong type $(2,2)$. 
For all balls $B$, let $A_B$ be a linear operator acting on $L^2(\Gamma)$.
Assume that, for all $j\geq 1$, there exists $\varphi(j) >0$ such that, for all $B \subset \Gamma$ and all functions supported in $B$ and all $j\geq 2$,
 \begin{equation} \label{Fen1inter3} \frac{1}{V(2^{j+1}B)^{\frac{1}{2}}} \left\| T(I-A_B)f\right\|_{L^2(C_j(B))} \leq \varphi(j) \dfrac{1}{V(B)^{\frac{1}{p_0}}} \|f\|_{L^{p_0}}\end{equation}
and for all $j\geq 1$
\begin{equation} \label{Fen1inter4} \frac{1}{V(2^{j+1}B)^{\frac{1}{p_0}}} \left\| A_Bf\right\|_{L^{2}(C_j(B))} \leq \varphi(j) \dfrac{1}{V(B)^{\frac{1}{p_0}}} \|f\|_{L^{p_0}}.\end{equation}
If $\ds \sum_{j\geq 1} \varphi(j) 2^{jd} <+ \infty$, where $d$ is given by Proposition \ref{Fen1PDV}, then $T$ is of weak type $(p_0,p_0)$, and therefore of strong type $(p,p)$ for all $p_0<p<2$. 

\end{theo}

\section{Littlewood-Paley functionals}

\label{Fen1LPtheory}

\subsection{$L^2(\Gamma)$-boundedness    of $ g^2_\beta$}
In order to prove Theorem \ref{Fen1maintheo}, let us introduce an extra functional.  
\begin{lem} \label{Fen1casL2}
Let $(\Gamma, \mu)$ be a weighted graph. Let $P$  be the operator defined by \eqref{Fen1defP}.

Define, for all $\beta>0$ and all functions $f\in L^2(\Gamma)$,     $g^2_\beta f$ by
 \[g^2_\beta f(x) = \left( \sum_{l\geq 1} b_l \left| (I-P^2)^{\beta} P^{l-1} f(x) \right|^2\right)^{\frac{1}{2}}\]
where $\sum_{l\geq 1} b_l z^{l-1}$ is the Taylor series of the function  $z \mapsto (1-z)^{-2\beta}$. 
Then $g^2_\beta$ is $L^2(\Gamma)$ bounded. More precisely, $g^2_\beta$ is an isometry on $L^2(\Gamma)$, which means that, for all $f\in L^2(\Gamma)$,
\[\| g^2_\beta f\|_{L^2(\Gamma)} = \|f\|_{L^2(\Gamma)}.\]
\end{lem}

\begin{proof}
Since $\|P\|_2 \leq 1 $, by spectral theory, $P$ can be written as
\[P = \int_{-1}^1 \lambda dE(\lambda).\]
It follows that for all $l\geq 1$, one has
\[(I-P^2)^{\beta} P^{l-1} = \int_{-1}^1 (1-\lambda^2)^{\beta}\lambda^{l-1} dE(\lambda)\]
so that, for all $f\in L^2(\Gamma)$ and $l\geq 1$,
\[\|(I-P^2)^{\beta}P^{l-1}f\|_{L^2}^2 = \int_{-1}^1 (1-\lambda^2)^{2\beta} \lambda^{2(l-1)} dE_{f,f}(\lambda).\]
The $L^2$-norm of   $g^2_\beta f$    can be now computed as
\[\begin{split}
   \|g^2_\beta f\|_{L^2}^2 & = \sum_{l\geq 1}b_l \|(I-P^2)^{\beta}P^{l-1}f\|_{L^2}^2 \\
& = \int_{-1}^1 (1-\lambda^2)^{2\beta} \sum_{l\geq 1} b_l \lambda^{2(l-1)} dE_{f,f}(\lambda) \\
& = \int_{-1}^1 dE_{f,f}(\lambda) \\
& = \|f\|_{L^2}^2
  \end{split}\]
where the third line is a consequence of the definition of $b_l$.
\end{proof}

\begin{lem}
 Let $(\Gamma, \mu)$ be a weighted graph satisfying \eqref{Fen1LB}.

Then $g_\beta$ and $\tilde g_\beta$ are  $L^2(\Gamma)$-bounded. 
\end{lem}

\begin{proof}
Since $\Gamma$ satisfies \eqref{Fen1LB}, $-1$ is not in the $L^2$ spectrum of $P$  (see for instance Lemma 1.3 in \cite{Dungey}).  Therefore there exists $a>-1$ such that
\[P = \int_{a}^1 \lambda dE(\lambda).\]
Proceeding as in the proof of the Lemma \ref{Fen1casL2}, we obtain
\[\begin{split}
   \|g_\beta f\|_{L^2}^2 & = \int_{a}^1 (1-\lambda)^{2\beta} \sum_{l\geq 1} l^{2\beta-1} \lambda^{2(l-1)} dE_{f,f}(\lambda) \\
& \lesssim \int_{a}^1 (1-\lambda)^{2\beta} \frac{1}{(1-\lambda^2)^{2\beta}} dE_{f,f}(\lambda) \\
& \quad = \int_{a}^1 \frac{1}{(1+\lambda)^{2\beta}} dE_{f,f}(\lambda) \\
& \lesssim \|f\|^2_{L^2}
  \end{split}\]
where, for the second line, we use Lemma \ref{Fen1DSE1}.\par
\noindent For $\tilde g_\beta$, just notice that, by definition of $\nabla$,
\[\|\tilde g_\beta f\|_{L^2} = \|g_{\beta+\frac{1}{2}} f\|_{L^2}.\]
\end{proof}

\subsection{$L^p(\Gamma)$-boundedness of $g_\beta$,   $2<p < + \infty$  }

The proof of the    $L^p$-   boundedness of $g_\beta$ for $p>2$ is based on the following Lemma and Theorem \ref{Fen1interpolation1}. 
The idea of the proof comes from Theorem 1.16 in \cite{BRuss}.

\begin{lem} \label{Fen1lemcalc}
 Let $(\Gamma,\mu)$ be a weighted graph satisfying \eqref{Fen1DV}, \eqref{Fen1LB} and \eqref{Fen1DUE}. 

For all $n \in \N^*$, there exists a constant   $C_n>0$    such that, for all balls $B = B(x_0,r)$ of $\Gamma$, all $j\geq 2$ and all $f$ supported in $C_j(B)$, one has
\[\left\| g_\beta (I-P^{r^2})^n f\right\|_{L^2(B)} \leq C_n 2^{j(\frac{d}{2}-2n)} \left( \frac{V(B)}{V(2^jB)} \right)^{\frac{1}{2}} \|f\|_{L^2}\]
\end{lem}

\begin{proof} First fix $n\in \N^*$. Denote by $\eta$ the only integer such that $\eta + 1 \geq \beta > \eta \geq 0$. We use the fact that
\[(I-P)^{\beta-1-\eta}  = \sum_{k\geq 0} a_k P^k\]
   where $\sum a_k z^k$ is the Taylor series    of the function $(1-z)^{\beta-\eta -1}$. 
Note that the equality   holds on $L^2(\Gamma)$ by    spectral theory and \eqref{Fen1Pcont}. Moreover, notice that if $\beta$ is an integer, then $a_k = \delta_0(k)$.

By the generalized Minkowski inequality, we get
\[ \begin{split}
    \left\| g_\beta (I-P^{r^2})^n f\right\|_{L^2(B)} & \leq \sum_{k\geq 0} a_k  \left( \sum_{l\geq 1} l^{2\beta-1} \left\| (I-P)^{1 + \eta}P^{k+l-1}(I-P^{r^2})^n f\right\|^2_{L^2(B)} \right)^{\frac{1}{2}}.
   \end{split}\]

We divide the sequel of the proof in 3 steps.
\begin{itemize}
 \item[{\bf 1-}]{\bf Estimate of the inner term}

Notice that $I-P^{r^2} = (I-P)\sum_{  s=0  }^{r^2-1} P^s$. Then, we get
\[\|(I-P)^{1+\eta}P^{k+l-1}(I-P^{r^2})^n f\|_{L^2(B)} \leq r^{2n} \sup_{s\in \bb 0,nr^2\bn} \|(I-P)^{1+\eta+n}P^{k+l+s-1 } f\|_{L^2(B)}\]
We   now estimate    the terms $\|(I-P)^{1+\eta}P^{k+l+s-1} f\|_{L^2(B)}$. For $0\leq s \leq nr^2$, since $f$ is supported in  $C_j(B)$ and by Remark \ref{Fen1pl-1}, one has,
\[\begin{split}
   \|(I-P)^{  1+n+\eta  }& P^{k+l-1 + s} f\|_{L^2(B)} \\ & \lesssim \frac{1}{(l+k+s)^{1+\eta+n}} \exp\left(-c\frac{(2^j-1)^2r^2}{l+k+s} \right) \|f\|_{L^2(C_j(B))} \\
& \lesssim \frac{1}{(l+k+s)^{1+\eta+n}} \exp\left(-c\frac{4^jr^2}{l+k+s} \right) \|f\|_{L^2(C_j(B))} \\
& \lesssim 2^\frac{jd}{2} \left(\frac{V(B)}{V(2^{j}B)} \right)^{\frac{1}{2}} \frac{1}{(l+k+s)^{1+\eta+n}} \exp\left(-c\frac{4^jr^2}{l+k+s} \right) \|f\|_{L^2}
  \end{split}\]
where the first line follows from \eqref{Fen1GT2}   and Cauchy-Schwarz    and the third one from \eqref{Fen1DV}.\par
\noindent Consequently, we obtain
\begin{equation} \begin{split} 
\|(I-P)^{1+\eta}P^{k+l-1}& (I-P^{r^2})^n f\|_{L^2(B)} \\ & \lesssim r^{2n} 2^{\frac{jd}{2}} \sup_{s\in \bb 0,nr^2\bn}\left[  \dfrac{\exp\left( -c \frac{(  4^j r^2  }{l+k+s}\right)}{(l+k+s)^{(1+\eta+n)}}\right] \left(\frac{V(B)}{V(2^jB)}\right)^{\frac{1}{2}} \|f\|_{L^2} \\
& \quad \leq  r^{2n} l^{-\eta} 2^{\frac{jd}{2}} \sup_{ s\in \bb 0,nr^2\bn }\left[  \dfrac{\exp\left( -c \frac{4^j r^2}{l+k+s}\right)}{(l+k+s)^{1+n}}\right] \left(\frac{V(B)}{V(2^jB)}\right)^{\frac{1}{2}} \|f\|_{L^2}.
\end{split}\end{equation}

\item[{\bf 2-}]{\bf   Reverse H\"older estimates   }\par
\noindent According to Proposition \ref{Fen1l2l1}    below   , the set of sequences $\{A^{k,r,j}_l, \, k \in \N, \, r \in \N^*, \, j\geq 2 \}$, where
\[A_l^{k,r,j} = l^{\beta-\eta} \sup_{s\in \bb 0,nr^2\bn} \left\{ \dfrac{\exp\left( -c \frac{4^j r^2}{l+k+s}\right)}{(l+k+s)^{1+n}} \right\},\]
is included in 
\[E_M = \left\{ (a_l)_{l\geq 1}, \ \forall l\in \N^*, \, 0 \leq a_l \leq M \sum_{k\in \N^*} \frac{1}{k} a_k \right\}\]  
for some $M>0$.  Therefore, Lemma \ref{Fen1L2L1} below yields   

\[\begin{split}
   V(B)^{-\frac{1}{2}} \left\| g_\beta (I-P^{r^2})^n f\right\|_{L^2(B)} & \lesssim r^{2n} 2^{\frac{jd}{2}}  V(2^jB)^{-\frac{1}{2}}\|f\|_{L^2} \sum_{k \geq 0} a_k \left( \sum_{l\geq 1} \frac{1}{l} (A_l^{k,r,j})^2\right)^{\frac{1}{2}} \\
& \lesssim r^{2n} 2^{\frac{jd}{2}} V(2^jB)^{-\frac{1}{2}}\|f\|_{L^2} \sum_{k \geq 0} a_k  \sum_{l\geq 1} \frac{1}{l} A_l^{k,r,j}.
 \end{split}\]

\item[{\bf 3-}]{\bf End of the calculus}\par

\noindent Note, thanks to Lemma \ref{Fen1DSE1}, that, when $\beta$ is not an integer, \label{Fen1RiemannIntegral}
\[ \begin{split} \sum_{k = 0}^{m-1} a_k   (m-k)  ^{\beta-1-\eta} &\lesssim 1 + \sum_{k = 1}^{m-1} k^{\eta -\beta}   (m-k)  ^{\beta-\eta-1} \\
& \quad = 1 + \frac{1}{m} \sum_{k=1}^{m-1} \left(\frac{k}{m}\right)^{\eta -\beta} \left(1-\frac{k}{m}\right)^{\beta-1-\eta} \\
& \quad \xrightarrow[m\to +\infty]{} 1 + \int_0^1 t^{\eta -\beta} (1-t)^{\beta-1-\eta}dt  < +\infty.
    \end{split} \]
The integral converges since    $\eta -\beta>-1$ and  $\beta-1-\eta>-1$. 
It follows that
\begin{equation} \label{Fen1ght}\sum_{k = 0}^{  m-1  } a_k   (m-k)  ^{\beta-\eta-1} \lesssim 1.\end{equation}
Since $a_k = \delta_0(k)$ and $\beta-1-\eta = 0$ when $\beta$ is an integer, the result above holds for all $\beta>0$.\par

   Using the expression of $A^{k,r,j}_l$,    we have {\small
\[\begin{split}
   V(B)^{-\frac{1}{2}} \left\| g_\beta (I-P^{r^2})^n f\right\|_{L^2(B)} & \lesssim V(2^jB)^{-\frac{1}{2}} r^{2n} 2^{\frac{jd}{2}} \|f\|_{L^2} \sum_{m=1}^{+\infty}   \sup_{s\in \bb 0,nr^2\bn} \left\{ \dfrac{\exp\left( -c \frac{4^j r^2}{m+s}\right)}{(m+s)^{1+n}} \right\}
 \end{split}\] }
  But, for some $c^{\prime}\in (0,c)$,
\[ \label{Fen1finRH}
\begin{split}
\sum_{m=1}^{+\infty} & \sup_{s\in \bb 0,nr^2\bn} \left\{ \dfrac{\exp\left( -c \frac{4^j r^2}{m+s}\right)}{(m+s)^{1+n}} \right\}  \\
& =  \frac 1{(4^jr^2)^{1+n}} \sum_{m=1}^{+\infty} \sup_{s\in \bb 0,nr^2\bn} \exp\left( -c \frac{4^j r^2}{m+s}\right)\left(\frac{4^jr^2}{m+s}\right)^{1+n}\\
&  \lesssim  \frac 1{(4^jr^2)^{1+n}} \sum_{m=1}^{4^jr^2} \exp\left( -c^{\prime} \frac{4^j r^2}{m+nr^2}\right)
 +  \frac 1{(4^jr^2)^{1+n}} \sum_{m=4^jr^2+1}^{+\infty} \left(\frac{4^jr^2}{m}\right)^{1+n}\\
&  \lesssim 4^{-jn} r^{-2n}.
\end{split}
\]

   The proof of Lemma \ref{Fen1lemcalc} is now complete.
\end{itemize}
\end{proof}

\begin{proof} 
   The proof of the $L^p$-boundedness of $g_\beta$ for $p>2$    is analogous to the one found in \cite{BRuss}, Theorem 1.16, when $2<p<+\infty$. Let us give the argument for the completeness.
We are aiming to use Theorem \ref{Fen1interpolation1}. It is enough to verify the validity of the assumptions \eqref{Fen1inter1} and \eqref{Fen1inter2}. We choose $A_B = I-(I-P^{r^2})^n$, where $r$ is the radius of $B$ and $n> \frac{d}{4}$.

\begin{itemize}
 \item[{\bf Proof of}] {\bf \eqref{Fen1inter1}} \par
\noindent We need to check that, for all $f\in L^2$, for all $x_0\in \Gamma$ and all balls $B \ni x_0$, one has
\[\frac{1}{V(B)^{\frac{1}{2}}} \|g_\beta (I-P^{r^2})^n f\|_{L^2(B)} \lesssim \left( \mathcal M |f|^2\right)^{\frac{1}{2}}(x_0).\]
We can decompose
\[f =  \sum_{j\geq 1} f\1_{C_j(B)} =: \sum_{j\geq 1} f_j.\]
First, since $g_\beta$ and $I-A_B = (I-P^{r^2})^n$ are   $L^2(\Gamma)$-bounded and by    the doubling property,
\[  \frac{1}{V(B)^{\frac{1}{2}}} \|g_\beta (I-P^{r^2})^n f_1\|_{L^2(B)} \lesssim \frac{1}{V(B)^{\frac{1}{2}}} \|f\|_{L^2(4B)} \lesssim \left( \mathcal M |f|^2\right)^{\frac{1}{2}}(x_0).   \]
For $j\geq 2$, Lemma \ref{Fen1lemcalc} provides:
\[
 \begin{split}
     \frac{1}{V(B)^{\frac{1}{2}}} \|g_\beta (I-P^{r^2})^n f_j\|_{L^2(B)}    & \lesssim 2^{j(\frac{d}{2}-2n)} \frac{1}{V(2^jB)^{\frac{1}{2}}} \|f_j\|_{L^2} \\
& \lesssim 2^{j(\frac{d}{2}-2n)} (\mathcal M|f|^2)^{\frac{1}{2}}(x_0).
 \end{split}
\]
Since $ n> \frac{d}{4}$, we can sum  in  $j\geq 1$, which gives the result.
\item[{\bf Proof of}] {\bf \eqref{Fen1inter2}} \par
\noindent What we have to show is that, for all $m\in \bb 1,n\bn$, all $f\in L^2(\Gamma) \cap L^\infty(\Gamma)$, all $x_0\in \Gamma$ and all balls $B\ni x_0$, one has,
\[\|g_\beta P^{mr^2} f\|_{L^\infty(B)} \lesssim (\mathcal M|g_\beta f|^2)^{\frac{1}{2}}(x_0).\]
First, since $\sum_{y\in G} p(x,y)m(y) = 1$, and by the use of Cauchy-Schwarz inequality, we obtain, for all $x\in \Gamma$ and $h\in L^2(\Gamma)$,
\[\left| P^{mr^2} h(x)\right| \leq \left( P^{mr^2} |h|^2(x)\right)^{\frac{1}{2}}.\]
Hence, it follows that for all $l\geq 1$
\[\left| P^{mr^2} (I-P)^{\beta}P^{l-1} f(x)\right|^2 \leq P^{mr^2}  |(I-P)^{\beta}P^{l-1} f|^2(x),\]
   so that, summing up  in  $l$,
\[\begin{split}
   (g_\beta P^{mr^2} f)(x)^2 &  =\sum_{l\geq 1} l^{2\beta-1} |P^{mr^2} (I-P)^{\beta}P^{l-1} f(x)|^2 \\
& \leq P^{mr^2} \left( \sum_{l\geq 1}    l^{2\beta-1} |(I-P)^{\beta}P^{l-1} f|^2\right)   (x) \\
& \quad = P^{mr^2} \left( |g_\beta f|^2 \right)(x) \\
& \lesssim \mathcal M \left( |g_\beta f|^2 \right)(x_0),
  \end{split}\]
   where the last line is due to \eqref{Fen1dominmax}.   
Here ends the proof of \eqref{Fen1inter2}, and the one of the    $L^p$-boundedness of $g_\beta$ for $p\in (2,+\infty)$.   
\end{itemize}
\end{proof}
\subsection{$L^p$-boundedness of $\tilde g_\beta$,    $2 \leq p < p_0$}

\begin{lem} \label{Fen1lemcalc3}
 Let $(\Gamma,\mu)$ be a weighted graph satisfying \eqref{Fen1DV}, \eqref{Fen1LB} and \eqref{Fen1DUE}. 

For all $n \in \N^*$, there exists a constant $C_n$ such that, for all balls $B = B(x_0,r)$ of $\Gamma$, all $j\geq 2$ and all $f$ supported in $C_j(B) = 2^{j+1}B \backslash 2^j B$, we get
\[\left\| \tilde g_\beta (I-P^{r^2})^n f\right\|_{L^2(B)} \leq C_n 2^{j(\frac{d}{2}-2n)} \left( \frac{V(B)}{V(2^jB)} \right)^{\frac{1}{2}} \|f\|_{L^2}.\]
\end{lem}

\begin{proof} (Lemma \ref{Fen1lemcalc3})
   
\noindent The proof is analogous to the one of Lemma \ref{Fen1lemcalc}, and we only indicates the main differences.   

 \noindent Define $\eta$ as in the proof of Lemma \ref{Fen1lemcalc}.    By the use of the generalized Minkowski inequality, we get
\[ \begin{split}
    \left\| \tilde g_\beta (I-P^{r^2})^n f\right\|_{L^2(B)} & \leq \sum_{k\geq 0} a_k  \left( \sum_{l\geq 1} l^{2\beta} \left\|\nabla(I-P)^{1 + \eta}P^{k+l-1}(I-P^{r^2})^n f\right\|^2_{L^2(B)} \right)^{\frac{1}{2}}.
   \end{split}\]

\noindent We now distinguish the cases $\beta>0$ ( {\it i.e. }  $\eta \in \N$) and $-\frac{1}{2} < \beta \leq 0$ ( {\it i.e. }  $\eta = -1$).

{\bf First case:} $\mathbf{\beta >0}$.    In this case, the proof is analogous to the one in Lemma \ref{Fen1lemcalc}, using \eqref{Fen1GGT2} instead of \eqref{Fen1GT2}.  

{\bf Second case:} $\mathbf{-\frac{1}{2}<\beta \leq 0.}$
\begin{itemize}
\item[$1.$]
By \eqref{Fen1GGT2}, {\small
\[\|\nabla P^{k+l-1} (I-P^{r^2})^n f\|_{L^2(B)} \lesssim 2^{\frac{jd}{2}} r^{2n} \hspace{-3pt} \sup_{s\in \bb 0,nr^2\bn} \left[  \dfrac{\exp\left( -c \frac{4^j r^2}{l+k+s}\right)}{(l+k+s)^{n+\frac{1}{2}}}\right] \left(\frac{V(B)}{V(2^jB)}\right)^{\frac{1}{2}} \|f\|_{L^2(C_j(B))}.  \]
}

\item[$2.$]
Define now  $B_l^{k,r,j}$    by
\[  B_l^{k,r,j} = l^{\beta+\frac{1}{2}}  \sup_{s\in \bb 0,nr^2\bn} \left\{ \dfrac{\exp\left( -c \frac{4^j r^2}{l+k+s}\right)}{(l+k+s)^{n+\frac{1}{2}}} \right\}\]
Remark \ref{Fen1careful} below therefore shows 
\begin{equation} \label{Fen1truc1} \begin{split}
   V(B)^{-\frac{1}{2}} \left\| \tilde g_\beta (I-P^{r^2})^n f\right\|_{L^2(B)} & \lesssim     2^{\frac{jd}{2}}V(2^jB)^{-\frac{1}{2}}\|f\|_{L^2} r^{2n}  \sum_{k \geq 0} a_k \left( \sum_{l\geq 1} \frac{1}{l} (B_l^{k,r,j})^2\right)^{\frac{1}{2}} \\
& \lesssim     2^{\frac{jd}{2}} V(2^jB)^{-\frac{1}{2}}\|f\|_{L^2}  r^{2n} \sum_{k \geq 0} a_k  \sum_{l\geq 1} \frac{1}{l} B_l^{k,r,j}.   
 \end{split} \end{equation} 
\item[$3.$]
Thanks to Lemma \ref{Fen1DSE1}, one has
\[ \begin{split} \sum_{k = 0}^{m-1} a_k (m-k)^{\beta-\frac{1}{2}} &\lesssim  m^{\beta-\frac{1}{2}} + \sum_{k = 1}^{m-1} k^{-\beta-1} (m-k)^{\beta-\frac{1}{2}} \\
& \lesssim   \frac{1}{\sqrt m} \int_0^1 t^{-\beta-1} (1-t)^{\beta-\frac{1}{2}}dt,
    \end{split} \]
if $\beta \in (-\frac{1}{2},0)$. If $\beta =0$, we have $a_k = \delta_0(k)$,    so that, in both cases,   
\begin{equation} \label{Fen1truc2} \sum_{k = 0}^{m-1} a_k (m-k)^{\beta-\frac{1}{2}} \lesssim \frac{1}{\sqrt m}.\end{equation} 
   Using \eqref{Fen1truc1} and \eqref{Fen1truc2}, one obtains   
\[\begin{split}
   V(B)^{-\frac{1}{2}} \left\| g_\beta (I-P^{r^2})^n f\right\|_{L^2(B)} & \\
 \lesssim 2^{\frac{jd}{2}} V(2^jB)^{-\frac{1}{2}} \|f\|_{L^2}  &  r^{2n} \sum_{m=1}^{+\infty}  \sup_{s\in \bb 0,nr^2\bn}  \left\{ \dfrac{\exp\left( -c \frac{4^j r^2}{m+s}\right)}{\sqrt m (m+s)^{n+\frac{1}{2}}} \right\}. 
 \end{split}\]
However, one has,
\[\begin{split}
   \sum_{m=1}^{+\infty}   \sup_{s\in \bb 0,nr^2\bn} & \left\{ \dfrac{\exp\left( -c \frac{4^j r^2}{m+s}\right)}{\sqrt m (m+s)^{n+\frac{1}{2}}} \right\} \\
& = \frac{1}{(4^jr^2)^{n+1}} \sum_{m=1}^{+\infty}  \frac{2^jr}{\sqrt{m}} \sup_{s\in \bb 0,nr^2\bn}   \left\{ \left( \frac{4^jr^2}{m+s}\right)^{n+\frac{1}{2}} \exp\left( -c \frac{4^j r^2}{m+s}\right) \right\}\\
& \lesssim \frac{1}{(4^jr^2)^{n+1}}  \sum_{m=1}^{4^jr^2}  \frac{2^jr}{\sqrt{m}} + \frac{1}{(4^jr^2)^{n+1}} \sum_{m=4^jr^2+1}^{+\infty}  \left(\frac{4^jr^2}{m}\right)^{n+1} \\
& \lesssim 4^{-jn} r^{-2n}.
  \end{split}\]

It yields the desired result
\[V(B)^{-\frac{1}{2}} \left\| g_\beta (I-P^{r^2})^n f\right\|_{L^2(B)} \lesssim  2^{j(\frac{d}{2}-2n)} V(2^jB)^{-\frac{1}{2}} \|f\|_{L^2} \]
\end{itemize}
\end{proof}

\begin{proof} (   $L^p$-boundedness    of $\tilde g_\beta$ for $2<p<p_0$)\par

\noindent We use Theorem \ref{Fen1interpolation1} as well.    The proof of \eqref{Fen1inter1} for $\tilde g_\beta$ is analogous to the corresponding one for $g_\beta$, by use of Lemma \ref{Fen1lemcalc3}.
Let us now check \eqref{Fen1inter2}. We argue as in \cite{ACDH} pp 932-936, using $(P_2)$ and ($GG_{p_0}$).    

We want to prove    that,    for all $2<p<p_0$, there exists $C_n$ such that for all balls $B\subset \Gamma$ of radius $r$, all $m\in \bb 0,n\bn$, all functions $f$ on $\Gamma$ and $x\in B$,
\begin{equation}
  \frac{1}{V^{\frac 1p}(B)}    \left\| \tilde g_\beta P^{2mr^2}f \right\|_{L^p(B)} \leq  
C_n \left(\mathcal M(|\tilde g_\beta f|^2)\right)^{\frac{1}{2}}(x).
\end{equation}

Let $f\in L^2(\Gamma)$. Since $P^l 1 \equiv 1$ for all $l\in \N$, we may write, if $g^l = (I-P)^{\beta} P^{l-1} f$,
\[\nabla P^{mr^2} (I-P)^{\beta} P^{l-1} f = \nabla P^{mr^2} \left(g^l - \left(g^l\right)_{4B}\right).\]
Write $\ds g^l - \left(g^l\right)_{4B} = \sum_{i\geq 1} g^l_i$ with $g^l_i = \left(g^l - \left(g^l\right)_{4B}\right)\1_{C_i(B)}$. 
For $i=1$,    Lemma $4.2$ in \cite{BRuss} and $(P_2)$ yield   
\[\begin{split}
 \left( \sum_{l\geq 1} l^{2\beta} \left( \frac{1}{V^{\frac{1}{p}}(B)} \|\nabla P^{mr^2} g_1^l\|_{L^p(B)}\right)^2 \right)^{\frac{1}{2}} & \lesssim \frac{1}{rV(4B)^{\frac{1}{2}}} \left( \sum_{l\geq 1} l^{2\beta} \|g_1^l\|^2_{L^2(4B)}\right)^{\frac{1}{2}} \\
& \lesssim \left( \frac{1}{V(8B)} \sum_{l\geq 1} l^{2\beta} \sum_{ y\in 8B }    |\nabla g_1^l(y)|^2m(y)   \right)^{\frac 12}\\
& \lesssim     \mathcal M_2 \left( \tilde g_\beta f \right)(x).   
\end{split}\]
For $i\geq 2$,    Lemma 4.2 in \cite{BRuss} shows that   
\[\left( \sum_{l\geq 1} l^{2\beta} \left( \frac{1}{V^{\frac{1}{p}}(B)} \|\nabla P^{mr^2} g_i^l\|\right)^2 \right)^{\frac{1}{2}}
\lesssim \frac{e^{-c4^i}}{r}  \left( \frac{1}{V(2^{i+1}B)} \sum_{l\geq 1} l^{2\beta} \|g^l_i\|_{L^2(C_i(B))}^2 \right)^{\frac{1}{2}}.
\]
But for all $l\geq 1$,
\[\begin{split}   \|g_i^l\|_{L^2(C_i(B))} & \leq \|g^l - \left(g^l\right)_{4B}\|_{L^2(2^{i+1}B)} \\
& \leq \|g^l - \left(g^l\right)_{2^{i+1}B}\|_{L^2(2^{i+1}B)} + V(2^{i+1}B)^{\frac{1}{2}} \sum_{j=2}^i |\left(g^l\right)_{2^{j}B}-\left(g^l\right)_{2^{j+1}B}|.
\end{split}\]
   For all $j\in \bb 2,i\bn$, $(P_2)$ implies
\[\begin{split} |\left(g^l\right)_{2^{j}B}-\left(g^l\right)_{2^{j+1}B}| & \lesssim  \frac{1}{V(2^{j+1}B)^{\frac{1}{2}}}  \|g^l - \left(g^l\right)_{2^{j+1}B}\|_{L^2(2^{j+1}B)} \\
   &  \lesssim 2^{j+1}r \frac{1}{V(2^{j+1}B)^{\frac{1}{2}}} \|\nabla g^l\|_{L^2(2^{j+1}B)}, \end{split}\]
while
\[
\|g^l - \left(g^l\right)_{2^{i+1}B}\|_{L^2(2^{i+1}B)}\lesssim 2^{i+1}r\left\Vert \nabla g^l\right\Vert_{L^2(2^{i+1}B)},
\]
so that
\[
 \|g_i^l\|_{L^2(C_i(B))}\lesssim \sum_{j=2}^{i} 2^{j}r\frac{V(2^{i+1}B)^{\frac{1}{2}}}{V(2^{j+1}B)^{\frac{1}{2}}}\left\Vert \nabla g^l\right\Vert_{L^2(2^{j+1}B)}.
 \]
As a consequence, by the Minkowski inequality, {\small
\[
\begin{array}{lll}
\displaystyle \left(\frac{1}{V(2^{i+1}B)}\sum_{l\geq 1} l^{2\beta} \|g^l_i\|_{L^2(C_i(B))}^2 \right)^{\frac{1}{2}} & \lesssim & \displaystyle \sum_{j=2}^{i} 2^jr\frac{1}{V(2^{j+1}B)^{\frac{1}{2}}}\left(\sum_{l\geq 1} l^{2\beta} \|\nabla g^l\|_{L^2(2^{j+2}B)}^2\right)^{\frac 12}\\
& \lesssim & \displaystyle \sum_{j=2}^{i} 2^jr {\mathcal M}_2\tilde{g}_{\beta}f(x)\\
& \lesssim & \displaystyle 2^{i}r{\mathcal M}_2\tilde{g}_{\beta}f(x).
\end{array}
\]}
\end{proof}

\subsection{ $L^p$-boundedness  of $g_\beta$  and $\tilde g_\beta$, $1 < p \leq 2$}

The proof of the    $L^p$-boundedness  of $g_\beta$ for $1<p<2$ relies on Theorem \ref{Fen1interpolation2}, via the following lemma:   

\begin{lem} \label{Fen1lemcalc2}
 Let $(\Gamma,\mu)$ be a weighted graph satisfying \eqref{Fen1DV}, \eqref{Fen1LB} and \eqref{Fen1DUE}. 

For all $n \in \N^*$, there exists a constant $C_n$ such that, for all balls $B = B(x_0,r)$ of $\Gamma$, all $j\geq 2$ and all $f \in L^{1}(\Gamma)$ supported in $B$, we get
\[\left\| g_\beta (I-P^{r^2})^n f\right\|_{L^2(C_j(B))} \leq C_n 2^{-2jn} \frac{V(2^jB)^{\frac{1}{2}}}{V(B)} \|f\|_{L^{1}}.\]
\end{lem}

\begin{proof}
The proof of Lemma \ref{Fen1lemcalc2} is very similar to the one of Lemma \ref{Fen1lemcalc}, and we will therefore by sketchy.    
First, we still have {\small
\[ \begin{split}
    \left\| g_\beta (I-P^{r^2})^n f\right\|_{L^2(C_j(B))} & \leq \sum_{k\geq 0} a_k  \left( \sum_{l\geq 1} l^{2\beta-1} \left\|(I-P)^{1+\eta} P^{k+l-1}(I-P^{r^2})^n f\right\|^2_{L^2(C_j(B))} \right)^{\frac{1}{2}}
   \end{split}\] }
where $a_k$ is defined as in the proof of \ref{Fen1lemcalc}. 
\begin{itemize}
\item[{\bf 1-}]{\bf Estimate of the inner term}

 Let $B = B(x_0,r)$.
   As in Lemma \ref{Fen1lemcalc} and using \eqref{Fen1GT2},   
\begin{equation} \begin{split}
\|(I-P)^{1+\eta}&  P^{k+l-1}(I-P^{r^2})^n  f\|_{L^2(C_j(B))} \\
& \lesssim l^{-\eta}  \|f\|_{L^{1}(B)} \sup_{s\in \bb 0,nr^2\bn}\left( \frac{1}{V(x_0,\sqrt{l+k+s})^{\frac{1}{2}}} \dfrac{\exp\left( -c \frac{4^j r^2}{l+k+s}\right)}{(l+k+s)^{1+n}}\right) \\
& \lesssim l^{-\eta} \frac{V(2^jB)^{\frac{1}{2}}}{V(B)} \|f\|_{L^{1}} \sup_{s\in \bb 0,nr^2\bn}\left(  \dfrac{\exp\left( -c \frac{4^j r^2}{l+k+s}\right)}{(l+k+s)^{1+n}}\right) \\
\end{split} \end{equation}
where we use for the second line the following fact, consequence of \eqref{Fen1DV}
\[  \frac{V(B)}{V(x_0,\sqrt{l+k+s})}   \lesssim \left(\frac{r^2}{l+k+s}\right)^{\frac{d}{2}} \lesssim \exp\left( -c \frac{4^j r^2}{l+k+s}\right).\]

\item[{\bf 2/3-}] {\bf Conclusion}\par
\noindent The proof is then the same (with obvious modifications) as the proof of Lemma \ref{Fen1lemcalc},    using the same sequence $A_l^{k,r,j}$ as in the proof of Lemma \ref{Fen1lemcalc}.   
\end{itemize}
\end{proof}

   We can now conclude for the $L^p$-boundedness of $g_\beta$ and $\tilde{g}_{\beta}$ for $1<p<2$.   

\begin{proof} (   $L^p$-boundedness and weak $(1,1)$ type of $g_\beta$ for $1<p<2$   )\par
\noindent We    apply    Theorem \ref{Fen1interpolation2}.    It is enough to check    \eqref{Fen1inter3} and \eqref{Fen1inter4} with $g(j)=2^{-j}$.    We take $A_B = P^{r^2}$ where $r$ is the radius of $B$. The inequality \eqref{Fen1inter3} is then a consequence of Lemma \ref{Fen1lemcalc2} for $n=1$. For the estimate \eqref{Fen1inter4},    it suffices to prove that,    for all balls $B$ of $\Gamma$, all $j\geq 1$, and all $f$ supported in $B$,
\[\|P^{r^2}f\|_{L^2(C_j(B))} \lesssim \frac{V(2^{j+1}B)^{\frac{1}{2}}}{V(B)}e^{-c4^j} \|f\|_{L^1(B)}.\]
The case $j\geq 2$ is a consequence of \eqref{Fen1GT2} and \eqref{Fen1DV}, while the case $j=1$ follows from \eqref{Fen1UE} and \eqref{Fen1ghk}.
\end{proof}

\begin{proof} (   $L^p$-boundedness and weak $(1,1)$ type of $\tilde g_\beta$   )\par
\noindent For $\beta>0$, the proof is the analogous to the one of the $L^p$-boundedness of $g_\beta$, using ($GGT_2$) instead of ($GT_2$).

The case $\beta \in \left( -\frac{1}{2},0\right]$ is analogous, with minor changes identical to the corresponding case in the proof of $L^p$-boundedness of $\tilde g_\beta$ for $p>2$. 

\end{proof}

\subsection{Reverse $L^p$ inequalities for    $g_\beta$ and $\tilde g_\beta$}

   Let us now end up the proof of Theorem \ref{Fen1maintheo}. What remains to be proved is:   
  \begin{theo} \label{Fen1theogeneral}
  Let $(\Gamma,\mu)$ be a weighted graph satisfying \eqref{Fen1DV}, \eqref{Fen1LB} and \eqref{Fen1DUE}. For all $1<p<+\infty$ and $\beta>0$, there exist three constants $C_1,C_2,C_3>0$ such that
  \[\|f\|_{L^p} \leq C_1 \|g_\beta f\|_{L^p} \leq C_2 \|g^2_\beta f\|_{L^p} \leq C_3 \|f\|_{L^p} \qquad \forall f\in L^p(\Gamma) \cap L^2(\Gamma).\]
  \end{theo}

\begin{rmq}
 Notice that Theorem \ref{Fen1theogeneral} implies Theorem \ref{Fen1maintheo} for $g_\beta$. A statement analogous to Theorem \ref{Fen1theogeneral} holds with $\tilde{g}_{\beta}$, with the same proof, which ends the proof of Theorem \ref{Fen1maintheo}.
\end{rmq}
\begin{proof}

By Lemma \ref{Fen1DSE1}, we get
\[g^2_\beta f(x) \simeq g_\beta (I+P)^{\beta} f(x) \qquad \forall f\in L^p(\Gamma) \cap L^2(\Gamma), \, \forall x\in \Gamma.\]
   As a consequence of this fact and Remark \ref{Fen1poweri-p}, for all $p\in (1,+\infty)$,    we have the inequalities
\begin{equation} \label{Fen1circle1} \|g^2_\beta f\|_{L^p} \lesssim \|g_\beta (I+P)^{\beta} f\|_{L^p} \lesssim \|(I+P)^{\beta} f\|_{L^p} \lesssim \|f\|_{L^p}.\end{equation}
The proof will then be complete if we establish, for all $1<p<+\infty$,
\begin{equation} \label{Fen1circle2} \|f\|_{L^p} \leq \|g^2_\beta f\|_{L^p} \qquad  \forall f\in L^p(\Gamma). \end{equation}
   Indeed, assume that \eqref{Fen1circle2} is established. The conjunction of    \eqref{Fen1circle1} and \eqref{Fen1circle2} provide the equivalences
 \[\|g^2_\beta f\|_{L^p} \simeq \|f\|_{L^p} \qquad \forall f\in L^p(\Gamma)\cap L^2(\Gamma)\]
and 
 \[\|g_\beta f\|_{L^p} \simeq \|f\|_{L^p} \qquad \forall f\in A = \{ (I+P)^{\beta} g, \ g\in L^p \cap L^2 \}\]
   and it is therefore enough to check that $A$ is dense in $L^p(\Gamma)$. \par
\noindent    To that purpose, notice     that \eqref{Fen1circle1} and \eqref{Fen1circle2} also provide the equivalence $\|(I+P)^{\beta}f\|_{L^p(\Gamma)} \simeq \|f\|_{L^p(\Gamma)}$    for all    $f\in L^2(\Gamma)\cap L^p(\Gamma)$, then    for all $f\in L^p(\Gamma)$ by the $L^p$-boundedness of $(I+P)^\beta$ and since $L^2(\Gamma)\cap L^p(\Gamma)$ is dense in $L^p(\Gamma)$. This entails that $(I+P)^{\beta}$ is one-to-one on $L^{p^{\prime}}(\Gamma)$ (with $\frac 1p+\frac 1{p^{\prime}}=1$), which implies that $A$ is dense in $L^p(\Gamma)$. \par  
\noindent The inequality \eqref{Fen1circle2} can be proven by duality. Actually,    for all $f,h\in L^2(\Gamma)$, Lemma \ref{Fen1casL2} shows that   
\[\begin{split}
   4<f,h> & = \|f+h\|_2^2 - \|f-h\|_2^2 \\
& = \|g^2_\beta (f+h)\|_2^2 - \|g^2_\beta (f-h)\|_2^2 \\
& \leq \|g^2_\beta f + g^2_\beta h\|_2^2 - \| g^2_\beta f- g^2_\beta h\|_2^2 \\
& \quad = 4<g^2_\beta f, g^2_\beta h>.
  \end{split}\]
For the   third line,    notice that
\[
g^2_\beta f-g^2_{\beta}h\leq g^2_{\beta}(f-h),
\]
and  interverting  the roles of $f$ and $h$, we obtain
\[
\left\vert g^2_\beta f-g^2_{\beta}h\right\vert \leq g^2_{\beta}(f-h),
\]
so that
\[
\left\Vert g^2_\beta f-g^2_{\beta}h\right\Vert_{L^2} \leq \left\Vert g^2_{\beta}(f-h)\right\Vert_{L^2}.
\]
Thus, if $\frac{1}{p} + \frac{1}{p^{\prime}} = 1$, we have for all $f\in L^p(\Gamma) \cap L^2(\Gamma)$, $1<p<+\infty$,
\[\begin{split}
   \|f\|_{L^p(\Gamma)} & = \sup_{\begin{subarray}{c} h\in L^2 \cap L^{p^{\prime}} \\ \|h\|_{  L^{p^{\prime}}  } \leq 1 \end{subarray}} <f,h> \\
& \leq \sup_{\begin{subarray}{c} h\in L^2 \cap L^{p^{\prime}} \\ \|h\|_{L^{p^{\prime}}} \leq 1 \end{subarray}} <g^2_\beta f, g^2_\beta h> \\
& \leq \|g^2_\beta f\|_{L^p} \sup_{\begin{subarray}{c} h\in L^2 \cap L^{p^{\prime}} \\ \|h\|_{L^{p^{\prime}}} \leq 1 \end{subarray}} \|g^2_\beta h\|_{L^{p^{\prime}}} \\
& \lesssim \|g^2_\beta f\|_{L^p} \sup_{\begin{subarray}{c} h\in L^2 \cap L^{p^{\prime}} \\ \|h\|_{L^{p^{\prime}}} \leq 1 \end{subarray}} \| h\|_{L^{p^{\prime}}} \\
& \quad = \|g^2_\beta f\|_{L^p}
  \end{split}\]
where the third line is a consequence of H\"older inequality and the fourth   one follows from    the boundedness of $g^2_\beta$ on $L^{p^{\prime}}(\Gamma)$. We obtain the desired result
\[\|f\|_{L^p} \lesssim \| g^2_\beta f\|_{L^p}.\]
\end{proof}

\section{$L^p$-boundedness of $\tilde g_0$, $1<p<2$  }

\noindent Define,    for all $q\in (1,2]$ and all functions $f$ on $\Gamma$,  
\[  \tilde N_q f := q f\Delta f - f^{2-q}\Delta f^q   \]
and,    for all functions $u_n:\N\times \Gamma\rightarrow \R$,   
\[  N_qu_n:= q u_n [\partial_n + \Delta] u_n - u_n^{2-q} [\partial_n + \Delta] u_n^q = \tilde N_q u_n + qu_n \partial_n u_n - u_n^{2-q} \dr_n u_n^q.   \]
   Here and after, $\partial_nu_n=u_{n+1}-u_n$ for all $n\in \N$.   
\begin{rmq}
 \begin{itemize}
  \item Dungey proved in \cite{Dungey2} that $0 \leq \tilde N_q(f) \leq \frac{q}{2} |\nabla f|^2$.
\item    The Young inequality shows at once that 
\begin{equation} \label{Fen1Young}
\dr_n u_n^q \geq qu_n^{q-1} \dr_n u_n,
\end{equation}
and then   $N_q(u_n) \leq \tilde N_q(u_n)$.   
  \item    As will be shown in Proposition \ref{Fen1propsup} below, $N_q(P^nf)\geq 0$ for all nonnegative functions $f$ and all $n\in \N$.    \end{itemize}
\end{rmq}

We    also introduce the functional   
\[\tilde g_{0,q} f(x) = \left( \sum_{n\geq 0} N_q(P_nf)(x) \right)^{\frac{1}{2}}.\]

\begin{theo} \label{Fen1secondtheo}
 If $q\in (1,2]$, then there exists a constant    $c>0$    such that
\[\|\tilde g_{0,q} f\|_q \leq c \|f\|_q\]
for all nonnegative functions $f \in L^1 \cap L^\infty$.
\end{theo}

\begin{cor} \label{Fen1NPnEstimates}
 Let $(\Gamma,\mu)$ be a graph satisfying \eqref{Fen1LB} and \eqref{Fen1LDV} and let $q\in (1,2]$ Then there exists $c_q>0$ such that
\[\|\nabla P^n f\|_q  \leq \frac{c_q}{\sqrt n} \|f\|_q\]
\end{cor}

\begin{rmq}
In \cite{Dungey2}, using semigroup arguments, Dungey proved the conclusion of Corollary \ref{Fen1NPnEstimates} under the   weaker    assumption that $-1$ does not belong to the $L^2$ spectrum of $P$. 
\end{rmq}

\subsection{Proof of Theorem \ref{Fen1secondtheo}}

The proof of this result is based on Stein's argument in \cite{Stein1}, Chapter II, also used in Riemannian manifolds in \cite{CDL} and on graphs with continuous time functionals in \cite{Dungey2}. \par
\noindent Let us first state the maximal ergodic theorem for Markov kernels (   see \cite{Stein2}, see also \cite{Stein1}, Chapter IV, Theorems 6 and 9  ):

\begin{lem} \label{Fen1ergodic} Let $(X,m)$ be a measurable space. 
Assume that $P$ is a linear operator  simultaneously  defined and bounded from $L^1(X)$ to itself and from $L^\infty(X)$ to itself that satisfies
\begin{enumerate}[i.]
 \item $P$ is self adjoint,
 \item $\|P\|_{L^1\to L^1} \leq 1$.
\end{enumerate}
 Let $f^*(x) = \sup_{n\geq 0} |P^n f(x)|$.
Then there exists a constant $c_q>0$ such that
\[\|f^*\|_q \leq c_q \|f\|_q\]
for all $q\in (1,+\infty]$.
\end{lem}

\noindent We can now turn to the proof of Theorem \ref{Fen1secondtheo}.

\noindent If $u_n = P^{n-1}f$, then $[\dr_n + \Delta]u_n = 0$ and,    as will be proved in Proposition \ref{Fen1propsup} below,    one has
\[N_q u_n = -u_n^{2-q} [\dr_n + \Delta] u_n^q \geq 0.\]

\noindent Consequently, we have
\[\begin{split}
   \tilde g_{0,q} f(x)^2 & =  \sum_{n \geq 0} N_q(P^{n} f)(x) = -\sum_{n\geq 0} [P^{n}f(x)]^{2-q} [\dr_n + \Delta] ([P^{n}f(x)]^q) \\
& \leq -f^*(x)^{2-q} \sum_{n\geq 0} [\dr_n + \Delta] ([P^{n}f(x)]^q).
  \end{split}\]
It follows, with $J(x) = -\sum_{n\geq 0} [\dr_n + \Delta] ([P^{n}f(x)]^q) \geq 0$,
\begin{equation}\label{Fen1sdf1}\begin{split}
   \|\tilde g_{0,q} f\|_q^q & \leq \sum_{x\in \Gamma} f^*(x)^{\frac{(2-q)q}{2}} J(x)^{\frac{q}{2}} m(x) \\
& \leq \left( \sum_{x\in \Gamma} f^*(x)^q m(x) \right)^{\frac{2-q}{2}} \left(\sum_{x\in \Gamma} J(x) m(x) \right)^{\frac{q}{2}}.
  \end{split}\end{equation}
Yet, by Lemma \ref{Fen1ergodic},
\begin{equation} \label{Fen1sdf2} \begin{split}
   \left( \sum_{x\in \Gamma} f^*(x)^q m(x) \right) \lesssim \|f\|_q^q
  \end{split}\end{equation}
and since $\ds \sum_{x\in \Gamma} \Delta g(x) m(x) = 0$ for all $g\in L^1(\Gamma)$,
\begin{equation} \label{Fen1sdf3}\begin{split}
   \sum_{x\in \Gamma} J(x) m(x) & = -\sum_{x\in \Gamma} m(x) \sum_{n\geq 0}\dr_n [P^{n}f(x)]^q \\
& \leq \sum_{x\in \Gamma} f(x)^q m(x) = \|f\|_q^q.
  \end{split}\end{equation}
The inequality in the last line is due to the fact that, for all $N\in \N$,
\[
\sum_{n=0}^N \dr_n \sum_{x\in \Gamma} [P^{n}f(x)]^qm(x)=\left\Vert f\right\Vert_q^q-\left\Vert P^{N+1}f\right\Vert_q^q.
\]

 Using \eqref{Fen1sdf1}, \eqref{Fen1sdf2} and \eqref{Fen1sdf3}, we thus obtain  the conclusion of Theorem \ref{Fen1secondtheo}.

\subsection{Proof of Theorem \ref{Fen1Maintheo}}

Recall   some    facts proved by Dungey in \cite{Dungey2}. Define the ``averaging'' operator $A$ by setting
\[(Af)(x) = \sum_{y\in B(x,2)} f(y)  = \sum_{y\sim x} f(y) \]
for $x\in \Gamma$ and functions $f:\Gamma\to \R$.

\begin{prop} \label{Fen1propDungey}
 Suppose that $(\Gamma,\mu)$ satisfies property \eqref{Fen1LDV}, and let $q\in (1,2]$. There exists $c_q >0$ such that
\[|\nabla f|^2 (x)  \leq c_q \, A(\tilde N_qf)(x)\]
for all    $x\in \Gamma$    and    all    nonnegative functions $f\in L^\infty$. Moreover, there exists $c'_q>0$ such that
\begin{equation} \label{Fen1averaging}
 \|AF\|_{\frac{q}{2}} \leq c'_q \|F\|_{\frac{q}{2}}
\end{equation}
for   all nonnegative functions    $F$ on $\Gamma$.
\end{prop}

\noindent Note that $\frac{q}{2} \leq 1$ in \eqref{Fen1averaging}, 
and that we use the notation $\|F\|_r:= \ds\left(\sum_{x\in \Gamma} m(x) |F(x)|^r\right)^{\frac{1}{r}}$ for $r\in (0,1]$.

\noindent In order to   prove    Theorem \ref{Fen1Maintheo}, we need the following result.

\begin{prop} \label{Fen1propsup}
Let $(\Gamma,\mu)$ be a weighted graph and let $q\in (1,2]$. Then $N_q(P^n f) \geq 0$ for all functions $0 \leq f\in  L^\infty$.

Moreover, if $(\Gamma,\mu)$ satisfies \eqref{Fen1LB}, there exists a constant $c_q>0$ such that
\[0 \leq \tilde N_q(P^n f) \leq c_q N_q(P^n f).\]
\end{prop}

\begin{proof} (Theorem \ref{Fen1Maintheo})

  Proposition \ref{Fen1propDungey} yields the pointwise estimate   
\[|\nabla P^n f|^2 \lesssim A( \tilde N_q (P^nf))\]
for $0\leq f \in L^\infty$,   so that, by Proposition \ref{Fen1propsup},   
\[\begin{split}
   (\tilde g_0  f)^2 & = \sum_{n\geq 0} |\nabla P^n f|^2 \\
& \lesssim \sum_{n\geq 0} A(N_q(P^nf)) \\
& \quad = A \left(\sum_{n\geq 0} N_q(P^n f)\right) = A \left(\tilde g_{0,q} f\right)^2.
  \end{split}\]
Theorem \ref{Fen1secondtheo} and \eqref{Fen1averaging} provide    the conclusion of    Theorem \ref{Fen1Maintheo} for all nonnegative functions $f$. 
We obtain then $L^q$-boundedness of $\tilde g_0$ by subadditivity of $\tilde g_0$. 
\end{proof}

It remains to prove Proposition \ref{Fen1propsup}.

\begin{proof} (Proposition \ref{Fen1propsup})

Taylor expansion of the function $t \mapsto t^q$,    $q\in (1,2]$,    gives
\begin{equation} \label{Fen1TaylorExp}
\begin{split}
t^q-s^q  &= q s^{q-1}(t-s) + q(q-1) \int_s^t \tau^{q-2} (t-\tau) d\tau \\
& = q s^{q-1}(t-s) + q(q-1)(t-s)^2 \int_0^1 \frac{(1-u)du}{((1-u)s+ut)^{2-q}} 
\end{split}
\end{equation}
for $t,s\geq 0$ with $s\neq t$. From this expansion, one has, for $q\in (1,2]$, $0\leq g\in L^\infty$ and $x\in \Gamma$, {\small
\begin{equation} \label{Fen1tildeNq}
\begin{split}
\tilde N_q(g)(x)  & = \sum_{y\in \Gamma} p(x,y)m(y) \left[qg(x)(g(x)-g(y))-g(x)^{2-q}(g(x)^q-g(y)^q) \right] \\
& = q(q-1) \sum_{y: \, g(y)\neq g(x)} p(x,y)m(y) (g(x)-g(y))^2 \int_0^1  \frac{(1-t)g(x)^{2-q}}{((1-t)g(x)+tg(y))^{2-q}}dt \\
& = q(q-1) g(x)^{2-q} \int_0^1 (1-t) \sum_{y: \, g(y)\neq g(x)} p(x,y)m(y) \frac{(g(x)-g(y))^2}{((g(x)+t(g(y)-g(x))^{2-q}}dt.
\end{split}
\end{equation} }
Let $0\leq f\in L^\infty$ and $n\in \N$. Define $g:=P^n f$ and notice $0\leq g\in L^\infty$. Therefore,
\[\begin{split}
   \dr_n (P^n f)(x) & = (P-I)g(x) = \sum_{y\sim x} p(x,y)m(y) (g(y)-g(x)) 
  \end{split}\]
and with \eqref{Fen1TaylorExp}, one has 
\[\begin{split}
   \dr_n (P^n f(x))^q - & q(P^n f(x))^{q-1} \dr_n (P^n f) \\ & = q(q-1)((P-I)g(x))^2 \int_0^1 \frac{(1-t)dt}{(g(x) + t (P-I)g(x))^{2-q}}
  \end{split}\] 
   so that   
\[\begin{split}
[\tilde N_q - N_q]g(x) & = g(x)^{2-q} \dr_n (P^n f(x))^q -    qg(x) \dr_n (P^n f)   \\
& = q(q-1) g(x)^{2-q}  \int_0^1 (1-t) \frac{((P-I)g(x))^2}{(g(x) + t (P-I)g(x))^{2-q}}dt
  \end{split} \]

If $g(x) = 0$, then $N_q(g)(x) = [\tilde N_q - N_q]g(x) = 0$,    therefore the conclusion of Proposition \ref{Fen1propsup} holds at $x$.     
Assume    now    that $g(x) \neq 0$.    Define, for all $y\in \Gamma$, $h(y) = \frac{g(y)-g(x)}{g(x)} \geq -1$ and,    for all $t\in (0,1)$ and all $s\in [-1,+\infty)$,    
\[\mathcal F_t (s) =\frac{s^2}{(1+ts)^{2-q}}.\] 
One has
\[\tilde N_q(g)(x) = q(q-1) g(x)^{2-q} \int_0^1 (1-t) \sum_{y: \, g(y)\neq g(x)} p(x,y)m(y) \mathcal F_t(h(y))dt\]
and
\[ [\tilde N_q - N_q]g(x) = q(q-1) g(x)^{2-q}  \int_0^1 (1-t) \mathcal F_t\left(\sum_{y: \, g(y)\neq g(x)} p(x,y)m(y)h(y)\right)dt.\]

   Assume for a while that it is known that $\mathcal F_t$ is convex on $[-1,+\infty)$ for all $t\in (0,1)$, and let us conclude the proof of Proposition \ref{Fen1propsup}.   
One has    $\tilde N_q(g)(x) \geq [\tilde N_q - N_q]g(x)$, which means that $N_q(g)(x) \geq 0$. 
Moreover,   since    $p(x,x) >\epsilon$, then $\sum_{y: \, g(y)\neq g(x)} \frac{p(x,y)m(y)}{1-\epsilon} \leq 1$,    so that    
\[\begin{split} \sum_{y: \, g(y)\neq g(x)} \frac{p(x,y)m(y)}{1-\epsilon} \mathcal F_t(h(y)) & \geq \mathcal F_t\left(\sum_{y: \, g(y)\neq g(x)} \frac{p(x,y)m(y)}{1-\epsilon}h(y)\right) \\ &\geq (1-\epsilon)^{-q} \mathcal F_t\left(\sum_{y: \, g(y)\neq g(x)} p(x,y)m(y)h(y)\right),\end{split}\]
where the first inequality is due to the convexity of ${\mathcal F}_t$ and the last one to the definition of ${\mathcal F}_t$. 
We deduce 
\[[\tilde N_q - N_q]g(x) \leq (1-\epsilon)^{q-1} \tilde N_q(g)(x),\]
which means
\[\tilde N_q(g)(x) \leq \frac{1}{1-(1-\epsilon)^{q-1}} N_q(g)(x).\]
\end{proof}

\noindent It remains to prove the following lemma

\begin{lem}
The function $\mathcal F_t$ is convex on $[-1,+\infty)$ for all $t\in (0,1)$.
\end{lem}

\begin{proof}
Let $F(x) = \frac{x^2}{(1+x)^{2-q}}$.    Easy computations show that $F$ is convex on $(-1,+\infty)$. Since, for all $t\in (0,1)$, $\mathcal F_t = \frac{1}{t^2} F(tx)$, $\mathcal F_t$ is convex on $(-\frac{1}{t}, +\infty) \supset [-1,+\infty)$.   
\end{proof}

\subsection{Proof of Corollary \ref{Fen1NPnEstimates}}

First we will prove the    following    result. If $q\in (1,2]$, $n\in \N^*$    with $n\geq 1$    and $0 \leq f \in L^1 \cap L^\infty$, one has
\begin{equation} \label{Fen1NqPnEstimates}
\|N_q^{\frac{1}{2}}(P^n f)\|_q \leq \frac{c_q}{\sqrt n} \|f\|_{q} 
\end{equation}

\noindent Let $u_n = P^n f$ and $J_n := - (\dr_n + \Delta)(u_n^q)$. Then 
\begin{equation} \label{Fen1NqPn1} \begin{split}
   \|N_q^{\frac{1}{2}} (P^n f)\|_q^q & =    \sum_{x\in \Gamma} m(x) N_q^{q/2}(u_n)(x)   \\
& = \sum_{x\in \Gamma} m(x) u_n^{\frac{q(2-q)}{2}} J_n(x)^{q/2} \\
& \leq \left[ \sum_{x\in \Gamma} m(x) u_n(x)^q \right]^{\frac{2-q}{2}} \left[ \sum_{x\in \Gamma} J_n(x)m(x) \right]^{\frac{q}{2}}
  \end{split} \end{equation}
where the last step follows from H\"olde inequality. 
Yet,
\[\sum_{x\in \Gamma} m(x) u_n(x)^q = \|P^n f\|_q^q \leq \|f\|_q^q\]
and
\[\begin{split}
   \sum_{x\in \Gamma} J_n(x)m(x) & = - \sum_{x\in \Gamma} \dr_n (u_n^q)(x)m(x) \\
& \leq -q \sum_{x\in \Gamma} m(x) u_n^{q-1}(x) \dr_n u_n(x) \\
& \leq q \|u_n\|_q^{q/q'} \|\dr_n u_n\|_q
  \end{split}\]
where the first line   holds    because $\sum_{x\in \Gamma}\Delta g(x)m(x)= 0$ if $g\in L^1$, the second line follows from \eqref{Fen1Young}, and the third one from H\"older inequality again (with $\frac{1}{q}+\frac{1}{q'}=1$).
Here $\|u_n\|_q \leq \|f\|_q$ while $\|\dr_n u_n \|_q = \|\Delta u_n\|_q  \lesssim \frac{1}{n}\|f\|_q$ by the analyticity of $P$ on $L^q$. Thus
\[\sum_{x\in \Gamma} J_n(x)m(x) \lesssim \frac{1}{n} \|f\|_q^q\]

\noindent Substitution of the last two estimates in \eqref{Fen1NqPn1} gives
\[\|N_q^{\frac{1}{2}} (P^n f)\|_q \lesssim \frac{1}{\sqrt n}\|f\|_q,\]
 which ends the proof of \eqref{Fen1NqPnEstimates}.

\noindent Now just use Propositions \ref{Fen1propDungey} and \ref{Fen1propsup} to get Corollary \ref{Fen1NPnEstimates}.

\appendix

\section{Further estimates for Markov chains}

\subsection{Time regularity estimates } \label{Fen1tools2}

The theorem we prove here is slightly more general than (and clearly implies) Theorem \ref{Fen1MTDUE}.

\begin{theo} \label{Fen1MTDUE2}
Let $(\Gamma, \mu)$ be a weighted graph satisfying \eqref{Fen1LB}, \eqref{Fen1DV} and \eqref{Fen1DUE}. Then, for all $j\geq 0$, there exist two constants $C_j,c_j>0$ such that, for all $(r_1, \dots, r_j) \in \N^{*j}$ for all $l\geq \max_{i\leq j} r_{i}$ and all $x,y\in \Gamma$,
\[
  |(D(r_1)\dots D(r_j)p)_l(x,y)| \leq \frac{C_jr_1\dots r_j}{l^j V(x,\sqrt l)^{\frac{1}{2}}V(y,\sqrt l)^{\frac{1}{2}}} \exp\left( -c_j \frac{d^2(x,y)}{l}\right).
\]
\end{theo}

We first recall the following result (Lemma 2.1 in \cite{Dungey}).

\begin{lem} \label{Fen1lemTDUE}
Let $P$ be a power bounded and analytic operator in a Banach space $X$. For each $j\in \N$ and $p\in (1,+\infty)$, there exists a constant $c_j>0$ such that
\[\|(I-P^{r_1})(I-P^{r_2})\dots(I-P^{r_j})P^l\|_{p\to p}\leq c_j r_1\dots r_j l^{-j}\]
for all $j_1, \dots, j_k \in \N$, all $(r_1, \dots, r_j)\in \N^j$ and all $l\in \N^*$.
\end{lem}

\begin{proof} let us now establish Theorem \ref{Fen1MTDUE2}.    We follow closely the proof of Theorem 1.1 of \cite{Dungey}, arguing by induction on $j$.\par
\noindent The case $j=0$ is obvious since the result is the assumption. 
\noindent The case $j=1$ and $r_1=1$ is the one proven by Dungey in \cite{Dungey} and we will just here verify that the proof for $j=1$ can be extended to all $j\in \N$.\par
\noindent Assume now that, for some $j\in \N$, the kernel $p_l$ satisfies for all $(r_1,\dots,r_j)\in \N^{*j}$ and all $l\geq \max_i {r_i}$
\begin{equation} \label{Fen1TDUEj}
  |D(r_j)\dots D(r_1)p_l(x,y)| \leq \frac{C_jr_1\dots r_j}{l^j V(x,\sqrt l)^{\frac{1}{2}}V(y,\sqrt l)^{\frac{1}{2}}} \exp\left( -c_j \frac{d^2(x,y)}{l}\right)
\end{equation}
where the constant $C_j$ depends only of the graph $\Gamma$ and   $j$.   \par
\noindent Let $(r_1,\dots,r_{j+1})\in \N^{*(j+1)}$. We then use the abstract identity (which can easily be proved by induction on $k$) for all linear operators $A$ and all $k\in \N$:\begin{equation} \label{Fen1Aformula}I-A = 2^{-(k+1)} (I-A^{2^{k+1}}) + \sum_{i=0}^k 2^{-(i+1)} (I-A^{2^i})^2 \end{equation}
where $I$ denotes the identity operator. Hence we have, applying \eqref{Fen1Aformula} with $(Au)_l = u_{l+r_{j+1}}$,
\[D(r_{j+1}) = 2^{-(k+1)} D(2^{k+1}r_{j+1}) + \sum_{i=0}^k 2^{-(i+1)} D(2^ir_{j+1})^2 \]
and if we apply this formula to $(D(r_j)\dots D(r_1)p)_l$, we obtain, for all $l\in \N$, $k\in \N$ and $x,y \in \Gamma$,
\begin{equation} \label{Fen1uio1} \begin{split}
|D(r_{j+1})\dots D(r_1)p_l(x,y)| & \leq 2^{-(k+1)} |D(2^{k+1}r_{j+1})D(r_{j})\dots D(r_1)p_l(x,y)| \\
& \qquad + \sum_{i=0}^k 2^{-(i+1)} |D(2^ir_{j+1})^2D(r_{j})\dots D(r_1)p_l(x,y)|.
\end{split}\end{equation}
Suppose that $0<2^k r_{j+1}\leq l$, hence $l+2^{k+1}r_{j+1} \leq 3l$ and \eqref{Fen1TDUEj} provides the estimate
\begin{equation} \label{Fen1uio2} \begin{split} 
   |D(2^{k+1}r_{j+1})& D(r_{j})\dots D(r_1)p_l(x,y)| \\ & \leq  |D(r_{j})\dots D(r_1)p_l(x,y)| +  |D(r_{j})\dots D(r_1)p_{l+2^{k+1}r_{j+1}}(x,y)| \\
& \leq C_j \frac{r_1\dots r_j}{l^j V(x,\sqrt l)^{\frac{1}{2}}V(y,\sqrt l)^{\frac{1}{2}}} \exp\left( -c_j \frac{d^2(x,y)}{3l}\right).  
 \end{split}\end{equation}
Besides, observe that 
\begin{equation} \label{Fen1uio4} \begin{split}
   |D(n)^2& D(r_{j}) \dots D(r_1)p_l(x,y)| \\ & = \|(I-P^n)^2(I-P^{r_j})\dots(I-P^{r_1})P^l\|_{L^1(\{y\})\to L^\infty(\{x\})} \\
& \leq \|P^{l_1}\|_{L^2\to L^\infty(\{x\})} \|(I-P^n)^2(I-P^{r_j})\dots(I-P^{r_1})P^{l_2}\|_{2\to2} \|P^{l_3}\|_{L^1(\{y\})\to L^2}.
  \end{split} \end{equation}
whenever $l=l_1+ l_2 +l_3$. Moreover, let us notice that for all $l_0\in \N^*$ and all $z\in \Gamma$, \eqref{Fen1DUE}   provides    
\begin{equation} \label{Fen1uio5} 
\begin{split} \|P^{l_0}\|_{L^2\to L^\infty(\{x\})} = \|P^{l_0}\|_{L^1(\{x\})\to L^2} & = \left( \sum_{y\in \Gamma} [p_{l_0}(x,y)]^2 m(y) \right)^{\frac{1}{2}} \\ & = p_{2l_0}(x,x)^{\frac{1}{2}} \\ & \leq \frac{C}{V(x,\sqrt l_0)^{\frac{1}{2}}}. \end{split}
\end{equation}

The two last results (\eqref{Fen1uio4} and \eqref{Fen1uio5})  combined with Lemma \ref{Fen1lemTDUE}  and the doubling property \eqref{Fen1DV} give, with $l_1,l_2,l_3 \sim \frac{l}{3}$
\begin{equation} \label{Fen1uio3}
 |D(n)^2D(r_{j})\dots D(r_1)p_l(x,y)| \leq C'_j \frac{n^2r_1\dots r_j}{l^{j+2} V(x,\sqrt l)^{\frac{1}{2}}V(y,\sqrt l)^{\frac{1}{2}}}.
\end{equation}
Collecting estimates \eqref{Fen1uio1}, \eqref{Fen1uio2} and \eqref{Fen1uio3} and using that $\ds\sum_{i=0}^k 2^{i-1} \leq 2^k$, we obtain
\[\begin{split}
   |D(r_{j+1})\dots D(r_1)p_l(x,y)| & \lesssim 2^{-(k+1)} \frac{r_1\dots r_j}{l^j V(x,\sqrt l)^{\frac{1}{2}}V(y,\sqrt l)^{\frac{1}{2}}} \exp\left( -c_j \frac{d^2(x,y)}{3l}\right) \\
& \qquad  + \sum_{i=0}^k 2^{-(i+1)} \frac{2^{2i}r_{j+1}^2r_j\dots r_1}{l^{j+2} V(x,\sqrt l)^{\frac{1}{2}}V(y,\sqrt l)^{\frac{1}{2}}} \\
& \lesssim 2^{-(k+1)} \frac{r_1\dots r_j}{l^j V(x,\sqrt l)^{\frac{1}{2}}V(y,\sqrt l)^{\frac{1}{2}}} \exp\left( -c_j \frac{d^2(x,y)}{3l}\right) \\
& \qquad + \frac{2^{k}r_1\dots r_j r_{j+1}^2}{l^{j+2} V(x,\sqrt l)} 
  \end{split}\]
for all $l\geq 1$, $k\in \N$ with $2^kr_{j+1} \leq l$.

We will now choose $k$   to obtain    the desired inequality. If  $l,j,x,y$  satisfy
\[l \exp\left(-c_j \frac{d^2(x,y)}{4l} \right) \geq r_{j+1}.\]
We choose $k$ such that
\[  2^kr_{j+1} \leq l \exp\left(-c_j \frac{d^2(x,y)}{4l} \right) <2^{k+1}r_{j+1}   \]
which gives
\[|D(r_{j+1})\dots D(r_1)p_l(x,y)| \lesssim \frac{r_1 \dots r_{j+1}}{l^{j+1} V(x,\sqrt l)^{\frac{1}{2}}V(y,\sqrt l)^{\frac{1}{2}}} \exp\left( -\frac{c_j}{12} \frac{d^2(x,y)}{l}\right).\]
In the other case, i.e. $l \exp\left(-c_j \frac{d^2(x,y)}{4l} \right) \leq r_{j+1}$, observe that by \eqref{Fen1TDUEj}
\[\begin{split}
   |D(r_{j+1})& \dots D(r_1)p_l(x,y)| \\ & \leq |D(r_{j})\dots D(r_1)p_l(x,y)| +  |D(r_{j})\dots D(r_1)p_{l+r_{j+1}}(x,y)| \\
& \leq C_j \frac{r_1 \dots r_j}{l^j V(x,\sqrt l)^{\frac{1}{2}}V(y,\sqrt l)^{\frac{1}{2}}} \exp\left( -c_j \frac{d^2(x,y)}{l+r_{j+1}}\right) \\
& \leq C_j \frac{r_1 \dots r_j}{l^j V(x,\sqrt l)^{\frac{1}{2}}V(y,\sqrt l)^{\frac{1}{2}}} \exp\left( -\frac{c_j}{2} \frac{d^2(x,y)}{l}\right) \\
& \leq C_j \frac{  r_1 \dots r_{j+1}  }{l^{j+1} V(x,\sqrt l)^{\frac{1}{2}}V(y,\sqrt l)^{\frac{1}{2}}} \exp\left( -\frac{c_j}{4} \frac{d^2(x,y)}{l}\right)
  \end{split}
\]
where the third line holds because $l\geq r_{j+1}$.
\end{proof}

\subsection{Gaffney-type inequalities}

\label{Fen1tools3}

  This paragraph is devoted to the proof of    Theorem \ref{Fen1GTineq}. Actually, we establish more general versions in Theorem \ref{Fen1GTinequalities} and Corollary \ref{Fen1GTinequalitiesbis}.   

\begin{theo} \label{Fen1GTinequalities}
Let $(\Gamma,\mu)$ be a weighted graph.  
Assume \eqref{Fen1LB}, \eqref{Fen1DV} and \eqref{Fen1DUE}. Then,     for all $j\in \N$, there exist $c,C>0$ such that for all $(r_1,\dots,r_j)\in \N^j$, for all sets $E,F\subset \Gamma$ and $x_0\in \Gamma$ such that $\sup\left\{d(x_0,y), \, y\in F \right\} \leq 3d(E,F)$ and all functions $f$ supported   in    $F$,
\begin{enumerate}[(i)]
\item $\displaystyle \| (I-P^{r_1})\dots(I-P^{r_j}) P^l f\|_{L^2(E)} \leq C \frac{r_1\dots r_j}{l^{j}} \frac{1}{V(x_0,\sqrt l)^{\frac{1}{2}}} e^{-c\frac{d(E,F)^2}{l}} \|f\|_{L^1(F)} $, 

for all $(r_1, \dots r_{j})\in \N^*$ and all $l\geq \max_{i\leq j} r_i$.
\item $\displaystyle \|\nabla (I-P^{r_1})\dots(I-P^{r_j}) P^l f\|_{L^2(E)} \leq C \frac{r_1\dots r_j}{l^{j+\frac{1}{2}}} \frac{1}{V(x_0,\sqrt l)^{\frac{1}{2}}} e^{-c\frac{d(E,F)^2}{l}} \|f\|_{L^1(F)}$, 

for all $(r_1, \dots r_{j})\in \N^*$ and all $l\geq \max_{i\leq j} r_i$.
\item $\displaystyle \|\nabla (I-P^{r_1})\dots(I-P^{r_j}) P^l f\|_{L^2(E)} \leq C \frac{r_1\dots r_j}{l^{j+\frac{1}{2}}} e^{-c\frac{d(E,F)^2}{l}} \|f\|_{L^2(F)}$, 

for all $(r_1, \dots r_{j})\in \N^*$ and all $l\geq \max_{i\leq j} r_i$.
\end{enumerate}
\end{theo}

\begin{cor} \label{Fen1GTinequalitiesbis}
Let $(\Gamma,\mu)$ be a weighted graph satisfying \eqref{Fen1LB}, \eqref{Fen1DV} and \eqref{Fen1DUE}. The conclusions of Theorem \ref{Fen1GTinequalities} still hold under any of    the following assumptions on $E,F,x_0$ and $l$:
\begin{enumerate}
 \item $\sup\left\{ d(x_0,y), \, y\in F \right\} \leq \sqrt l$,
 \item $\sup\left\{ d(x_0,x), \, x\in E \right\} \leq 3 d(E,F)$,
 \item $\sup\left\{ d(x_0,x), \, x\in E \right\} \leq \sqrt l$.
\end{enumerate}
\end{cor}

The proof of Theorem \ref{Fen1GTinequalities} relies on:   

\begin{lem}  Let $(\Gamma, \mu)$ be a weighted graph satisfying \eqref{Fen1DV}, \eqref{Fen1LB} and \eqref{Fen1DUE}, then we have the following estimates: 
for all $j\in \N$, there exist $C_j,c_j>0$ such that for all $(r_1, \dots r_j)\in \N^{*j}$ and all $l\geq max_{i\leq j} r_i$
\begin{equation} \label{Fen1lemE} \sum_{y\in \Gamma} |\left(D(r_j)\dots D(r_1) p\right)_l(y,x)|^2 e^{c_j\frac{d(x,y)^2}{l}} m(y) \leq C_{j} \frac{  r_1^2\dots r_j^2  }{l^{2j} V(x,\sqrt l)}\end{equation}
and
\begin{equation} \label{Fen1lemGE} \sum_{y\in \Gamma} |\nabla_y \left(D(r_j)\dots D(r_1) p\right)_l(y,x)|^2 e^{c_j\frac{d(x,y)^2}{l}} m(y) \leq C_{j} \frac{  r_1^2\dots r_j^2  }{l^{2j+1} V(x,\sqrt l)}. \end{equation}
\end{lem}

The proof of this Lemma is analogous to Lemmas 4 and 7 in \cite{Russ},  where we use the estimates in Theorem \ref{Fen1MTDUE2} instead of the estimate \eqref{Fen1UE}.

\begin{proof} (Theorem \ref{Fen1GTinequalities})

\noindent  \begin{enumerate}[(i)]
\item We can assume without loss of generality that $\|f\|_{L^1} = 1$. Then 
\[\begin{split}
  &\| (I-P^{r_1}) \dots (I-P^{r_j})P_l f\|_{L^2(E)}^2 \\
& = \sum_{x\in E} m(x) \left(\sum_{z\in F} \left(D(r_j)\dots D(r_1) p\right)_l(x,z)f(z)m(z) \right)^2 \\
& \leq \sum_{x\in E} m(x) \sum_{z\in F} |\left(D(r_j)\dots D(r_1) p\right)_l(x,z)|^2 |f(z)|m(z) \\
& \leq \exp\left( -c \frac{d(E,F)^2}{l} \right) \sum_{z\in F} |f(z)| m(z) \sum_{x\in E} m(x) |D(r_j)\dots D(r_1) p_l(x,z)|^2 \exp\left(c\frac{d(x,z)^2}{l}\right) \\
& \lesssim \frac{r_1^2\dots r_j^2}{l^{2j}} \exp\left( -c \frac{d(E,F)^2}{l} \right) \sum_{z\in F} |f(z)| m(z) \frac{1}{V(z,\sqrt l)} \\
& \lesssim \frac{r_1^2\dots r_j^2}{l^{2j} V(x_0,\sqrt l)} \exp\left( -c \frac{d(E,F)^2}{l} \right) 
  \end{split}\] 
where, for the $4^{th}$ line, we use  the estimate \eqref{Fen1lemE}  and, for the last line, the doubling property shows   
\begin{equation} \label{Fen1ghk} \frac{V(x_0,\sqrt l)}{V(z,\sqrt{l})} \leq \frac{V(z,\sqrt l +3d(E,F))}{V(z,\sqrt l)} \lesssim  \left( 1+ \frac{3d(E,F)}{\sqrt{l}}\right)^d \lesssim \exp\left(\frac{c}{2} \frac{d(E,F)^2}{l} \right) \end{equation}
which leads to the result (with a different value    of $c$).

\item Similar to (i) using \eqref{Fen1lemGE} instead of \eqref{Fen1lemE}.
\item This result is a consequence of (i). In fact,
\[\begin{split}
   \| \nabla (I-P^{r_1})\dots (I-P^{r_j}) P^l f\|_{L^2(E)} & \lesssim \frac{r_1\dots r_j}{l^{j+\frac{1}{2}}} \frac{1}{V(x_0,\sqrt l)^{\frac{1}{2}}} e^{-c\frac{d(E,F)^2}{l}} \|f\|_{L^1(F)} \\
& \quad \leq \frac{r_1\dots r_j}{l^{j+\frac{1}{2}}} \left(\frac{m(F)}{V(x_0,\sqrt l)}\right)^{\frac{1}{2}} e^{-c\frac{d(E,F)^2}{l}} \|f\|_{L^2(F)} \\
& \lesssim \frac{r_1 \dots r_j}{l^{j+\frac{1}{2}}} e^{-c\frac{d(E,F)^2}{l}} \|f\|_{L^2(F)}
  \end{split}\]
where, for the last line, the doubling property yields   
\begin{equation} \label{Fen1ghj} \dfrac{m(F)}{V(x_0,\sqrt{l})} \leq \dfrac{V(x_0,3d(E,F))}{V(x_0,\sqrt{l})} \lesssim \left(1+\frac{3d(E,F)}{\sqrt l}\right)^{d} \lesssim \exp\left(\frac{c}{2} \frac{d(E,F)^2}{l} \right) . \end{equation}
\end{enumerate}

\end{proof}

\begin{proof} (of Corollary \ref{Fen1GTinequalitiesbis})
\begin{enumerate}
 \item Under this assumption, the proof is analogous to the one of Theorem \ref{Fen1GTinequalities}, replacing \eqref{Fen1ghk} by 
\[\frac{V(x_0,\sqrt l)}{V(z,\sqrt{l})} \lesssim 1 \qquad \forall z\in F,\] 
which is provided by \eqref{Fen1DV} and \eqref{Fen1ghj} by
\[\frac{m(F)}{V(x_0,\sqrt{l})} \leq 1,\]
which is due to the fact that $F\subset B(x_0,\sqrt{l})$.   
\item Decompose $F = \ds \bigcup_{i\geq 0} F_i$, with
\[F_i = F \cap \{y\in \Gamma, \, 3^id(E,F)\leq d(y,E) < 3^{i+1}d(E,F)\}.\]
Remark that, if $F_i\neq \emptyset$,    
\[\sup_{y\in F_i} d(x_0,y) \leq (3+3^{i+1}) d(E,F) \leq (3+ 3^{1-i})d(E,F_i) \leq 6 d(E,F_i).\]
\noindent Let $T$ be one of the operators involved in the left-hand sides    in Theorem \ref{Fen1GTinequalities}. 
Let $c_T>0$ be such that,    for all $(\tilde E, \tilde F)$ such that $\ds \sup_{y\in \tilde F} d(x_0,y) \leq 6 d(\tilde E,\tilde F)$ and all $f$ supported in $\tilde F$, we have
\[\|Tf\|_{L^2(\tilde E)} \leq c_T e^{-c\frac{d(\tilde E,\tilde F)^2}{l}} \|f\|_{L^1(\tilde F)}.\]
(Remember that Theorem \ref{Fen1GTinequalities} can be proven with  constant $6$  instead of $3$.) Then, one has
\[\|Tf\|_{L^2(E)} \leq \sum_{i=0}^{+\infty} \|T(f\1_{F_i})\|_{L^2(E)} \leq c_T \sum_{i=0}^{+\infty} e^{-c\frac{ d(E,F_i)^2}{l}} \|f\|_{L^1(F_i)} \leq c_T e^{-c\frac{d(E,F)^2}{l}} \|f\|_{L^1(F)},\]
which proves the second point of the corollary (note that the above sum can be restricted to the indexes $i$ such that $F_i\neq \emptyset$).
\item Let $R = \sup_{x\in E} d(x,x_0)$. Decompose $F= \ds \bigcup_{i\geq 1} F_i$ with
\[F_1 = F \cap B(x_0,4R)\]
and if $i\geq 2$
\[F_i = F \cap \{y\in \Gamma, \, 2^i R <d(y,x_0) \leq 2^{i+1}R \}.   \]

Write 
\[\|Tf\|_{L^2(E)} \leq \sum_{i=1}^{+\infty} \|T(f\1_{F_i})\|_{L^2(E)}\]
(where $T$ is one of the sublinear operators of Theorem \ref{Fen1GTinequalities} and $f$ is supported in $F$). We want to estimate each $\|T(f\1_{F_i})\|_{L^2(E)}$. First, notice that $\ds \sup_{x\in F^1} \{d(x,x_0)\} \leq 4R \leq 4\sqrt l$. Use then point 1 of Corollary \ref{Fen1GTinequalitiesbis} to obtain
\[\begin{split}
   \|T(f\1_{F_1})\|_{L^2(E)} & \leq c_T e^{-c\frac{d(E,F_1)^2}{l}} \|f\|_{L^1(F_1)} \\
& \leq c_T e^{-c\frac{d(E,F)^2}{l}} \|f\|_{L^1(F_1)}.
  \end{split}\]
  Next,    remark that, if $i\geq 2$ and since $d(E,F_i) \geq (2^i-1)R$,
\[\sup_{x\in F_i} \{d(x,x_0)\} \leq 2^{i+1}R \leq \frac{2^{i+1}}{2^i-1} d(E,F_i) \leq 3 d(E,F_i).\]
Hence, using Theorem \ref{Fen1GTinequalities}, one has
\[\begin{split}
   \|T(f\1_{F_i})\|_{L^2(E)} & \leq c_T e^{-c\frac{d(E,F_i)^2}{l}} \|f\|_{L^1(F_i)} \\
& \leq c_T e^{-c\frac{d(E,F)^2}{l}} \|f\|_{L^1(F_i)}.
  \end{split}\]
\end{enumerate} 
Summing up over $i$ yields the desired conclusion.   
\end{proof}

\section{Estimates for the Taylor coefficients of $(1-z)^{-\beta}$}

\begin{lem} \label{Fen1DSE1}
Let $\gamma >-1$. Let $\ds \sum_{l \geq 1} a_l z^l$ be the Taylor series of the function $\frac{1}{(1-z)^{\gamma + 1}}$. We have 
\[a_l \simeq  l^{\gamma} \qquad \forall l\in \N^*.\]
A consequence of this result is
\[\sum_{l \geq 0} l^{\gamma} z^l \simeq \dfrac{z}{(1-z)^{\gamma + 1}} \qquad  \forall z\in [0,1).\]
\end{lem}

\begin{proof}
For $|z|<1$, the holomorphic function $\frac{1}{(1-z)^{\gamma + 1}}$ is equal to    its Taylor series,   
\[\dfrac{1}{(1-z)^{\gamma + 1}} = \sum_{l\geq 0}  \prod_{i=1}^l \left(1+\frac{\gamma}{i}\right) z^l \qquad \forall z\in B_{\C}(0,1).\]
Let us check that
\begin{equation} \label{Fen1fgh} l^\gamma \simeq  \prod_{i=1}^l \left(1+\frac{\gamma}{i}\right) \qquad \forall l\geq 1 .\end{equation}
Indeed, one can write
\[\begin{split}
   \ln \prod_{i=1}^l \left(1+\frac{\gamma}{i}\right) & = \sum_{i=1}^l \ln \left(1+\frac{\gamma}{i}\right)  \\
& = \gamma \sum_{i=1}^l \frac{1}{i} + \sum_{i=1}^l \left[ \ln \left(1+\frac{\gamma}{i}\right) - \frac{\gamma}{i}.\right]
  \end{split}\]
Yet, one has
\[\begin{split}
   \left| \sum_{i=1}^l \left[ \ln \left(1+\frac{\gamma}{i}\right) - \frac{\gamma}{i} \right] \right| & \leq \sum_{i=1}^\infty \left| \ln \left(1+\frac{\gamma}{i}\right) - \frac{\gamma}{i} \right| \\
& \lesssim \sum_{i=1}^\infty \frac{\gamma^2}{i^2} <+\infty.
  \end{split}\]
Hence we get
\[ \ln \prod_{i=1}^l \left(1+\frac{\gamma}{i}\right) = \gamma \ln l + O(1),\]
which yields \eqref{Fen1fgh} by applying the exponential map.\par
\noindent From the last result, and since the convergence radius of the series under consideration are 1, we deduce
\[\begin{split}
   \sum_{l \geq 1} l^\gamma z^l & \simeq \sum_{l\geq 1}  \prod_{i=1}^l \left(1+\frac{\gamma}{i}\right) z^l \qquad \forall z\in [0,1) \\
& \quad = \dfrac{1}{(1-z)^{\gamma + 1}} -1 \\
& \quad = \dfrac{1 - (1-z)^{\gamma + 1}}{(1-z)^{\gamma + 1}} \\
& \simeq \dfrac{z}{(1-z)^{\gamma + 1}}.
  \end{split}\]
\end{proof}

\section{Reverse Hölder estimates \\ for sequences}

For all $M>0$, define the following sets of sequences
\[E_M = \left\{ (a_n)_{n\geq 1}, \ \forall n, \, 0 \leq a_n \leq M \sum_{k\in \N^*} \frac{1}{k} a_k \right\}\]
and
\[\tilde E_M = \left\{  (a_n)_{n\geq 1} , \ \forall n, \, 0 \leq a_n \leq M \sum_{k\geq n} \frac{1}{k} a_k \right\}.\]

First, we state this obvious lemma:
\begin{lem} \label{Fen1L2L1}
\[\left( \sum_{n\geq 1} \frac{1}{n} a_n^2 \right)^{\frac{1}{2}} \leq M^{\frac{1}{2}} \sum_{n\geq 1} \frac{1}{n} a_n  \qquad \forall (a_n)_n \in E_M. \]
\end{lem}

\noindent Let $\mathcal A = \{(A^{k,r,j}_l)_{l\in \N^*}, \, k \in \N, \, r \in \N^*, \, j\geq 2 \}$, where
\[A_l^{k,r,j} = l^{\beta-\eta} \sup_{s\in \bb 0,nr^2\bn} \left\{ \dfrac{\exp\left( -c \frac{4^j r^2}{l+k+s}\right)}{(l+k+s)^{1+n}} \right\}.\]
The parameters $\beta$ and $\eta$ are chosen as in section \ref{Fen1LPtheory} and therefore $\beta-\eta \in (0,1]$.

\begin{prop} \label{Fen1l2l1}
 There exists $M>0$ such that $\mathcal A \subset E_M$.
\end{prop}

In order to prove Proposition \ref{Fen1l2l1}, we will need the following Lemmata:

\begin{lem} \label{Fen1l2l1b}
One has the next three  results:
\begin{enumerate}[i.]
\item $\tilde E_M \subset E_M,$
\item  if $M>0$ and $\{(a_n^p)_n, \ p\in I\}$ is a set of sequences such that for all $p\in I$, $(a^p_n)_n\in \tilde E_M$, then $\ds\left(\sup_{p\in I} a_n^p\right)_n \in \tilde E_M$, 
\item For a positive sequence $\lambda$, we define,   for all sequences $a\in \R^N$, $\rho_\lambda(a)$ by   
\[ [\rho_\lambda(a)]_n = \frac{\lambda_n}{\lambda_{n+1}} a_{n+1}.\]
Then, if $(\lambda_n)_n$ is non decreasing and $\ds \left( \frac{\lambda_n}{n}\right)_n$ is non increasing, $\tilde E_M$ is stable by $\rho_\lambda$.
\end{enumerate}
\end{lem}

\begin{proof} (i) and (ii) are easy to prove. Let us check (iii).

Since $(a_n)_n \in \tilde E_M$, we have for all $n\in \N^*$
\[\begin{split}[\rho_\lambda(a)]_{n} & = \frac{\lambda_n}{\lambda_{n+1}} a_{n+1} \\
& \leq M \frac{\lambda_n}{\lambda_{n+1}} \sum_{k \geq n + 1} \frac{1}{k} a_{k} \\
& \quad = M \frac{\lambda_n}{\lambda_{n+1}} \sum_{k \geq n} \frac{1}{k+1} a_{k+1} \\
& \quad = M \frac{\lambda_n}{\lambda_{n+1}} \sum_{k\geq n} \frac{1}{k} \frac{k\lambda_{k+1}}{(k+1)\lambda_k}  [\rho_\lambda(a)]_k \\
& \leq M \sum_{k\geq n} \frac{1}{k} [\rho_\lambda(a)]_{k}
\end{split}\]
because $\left(\frac{\lambda_k}{k}\right)_k$ is non increasing and $(\lambda_n)_n$ is non decreasing.
\end{proof}

Define for $c,\alpha >0$,
 \[\begin{split} A^\alpha_c = \left\{ (a_n)_{n\in \N^*},\right. & \  \exists n_0\in \N^* \text{ such that } \forall n< n_0, \, a_n\leq a_{n+1} \\ & \left. \text{ and } \forall n\geq n_0, \, ca_{n_0} \left( \frac{n_0}{n} \right)^\alpha \leq a_n \leq \frac{1}{c}a_{n_0} \left( \frac{n_0}{n} \right)^\alpha\right\}.\end{split}\]

\begin{lem} \label{Fen1ssensEM}
For all $(a_n)_n \in A^\alpha_c$ and all $n\geq n_0$ (where $n_0$ is given by the def of $A^\alpha_c$),
\[ a_{n} \simeq \sum_{k \geq n} \frac{1}{k} a_k.\]
In particular, there exists $M$ (only depending on $\alpha$ and $c$) such that $A^\alpha_c \subset \tilde E_M$.
\end{lem}

\begin{proof}
One has if $(a_n)_n \in A^\alpha_c$ and $n\geq n_0$ ,
\[\begin{split}
  \sum_{k\geq n} \frac{1}{k} a_k & \simeq a_{n_0} n_0^\alpha  \sum_{k\geq n} \frac{1}{k^{\alpha+1}} \\
& \simeq  a_{n_0} \left(\frac{n_0}{n}\right)^{\alpha} \\
& \simeq a_n.
\end{split}\]
\end{proof}

We are now ready for the proof of Proposition \ref{Fen1l2l1}.

\begin{proof} (of Proposition \ref{Fen1l2l1})\par
\noindent According to Lemma \ref{Fen1l2l1b} and Lemma \ref{Fen1ssensEM}, we only need to prove that 
\[\mathcal A_0 = \left\{\left( l^{\beta-\eta} \dfrac{\exp\left( -c \frac{4^j r^2}{l}\right)}{l^{1+n}}\right)_{l\in \N^*}, \, r \in \N^*, \, j\geq 2 \right\}\]
 is in some $A_c^\alpha$. 
Indeed, once we proved $\mathcal A_0 \subset A_c^\alpha$, Lemma \ref{Fen1ssensEM} implies that there exists $M>0$ such that $\mathcal A_0 \subset \tilde E_M$. 
The use of Lemma \ref{Fen1l2l1b}(iii) with $\lambda_l = l^{\beta -\eta}$ yields, since $\beta - \eta \in (0,1]$,
\[\mathcal A_1 : = \bigcup_{k\in \N} (\rho_\lambda)^k(\mathcal A_0) = \left\{\left( l^{\beta-\eta} \dfrac{\exp\left( -c \frac{4^j r^2}{l+k}\right)}{(l+k)^{1+n}}\right)_{l\in \N^*}, \, r \in \N^*, \, k\in \N,\, j\geq 2 \right\} \subset \tilde E_M.\]
Lemma \ref{Fen1l2l1b}(ii) thus provides that $\mathcal A \subset \tilde E_M$ and Lemma \ref{Fen1l2l1b}(i) that $\mathcal A \subset E_M$.

It remains to prove that $\mathcal A_0 \subset A_c^\alpha$.
The result is a consequence of the following facts.

 For $\gamma\in [0,1]$ and $n\geq 1$, the function
\[F: t\in \R_+ \mapsto t^{\gamma} \dfrac{\exp\left( - \frac{d}{t}\right)}{t^{n+1}} \] 
satisfies
\begin{itemize}
 \item   $F(0)=0$ and $\ds \lim_{t\rightarrow +\infty} F(t)=0$,   
 \item   $F$ reaches its unique maximum at  $t_0= \frac{d}{n+1-\gamma}$,    
  \item $\ds \frac{e^{\gamma-n-1}}{t^{n+1-\gamma}} \leq F(t) \leq \frac{1}{t^{n+1-\gamma}}$ for all $t\geq t_0$. 
\end{itemize}
\end{proof}

\begin{rmq} \label{Fen1careful}
If $\beta\in \left(-\frac 12,0\right]$ and
\[B_l^{k,r,j} = l^{\beta+\frac{1}{2}} \sup_{s\in \bb 0,nr^2\bn} \left\{ \dfrac{\exp\left( -c \frac{4^j r^2}{l+k+s}\right)}{(l+k+s)^{n+\frac{1}{2}}} \right\},\]
a careful inspection of the proof of Proposition \ref{Fen1l2l1} shows that the conclusion of Proposition \ref{Fen1l2l1} also holds for $B_l^{k,r,j}$.
\end{rmq}

\bibliographystyle{plain}
\bibliography{biblio}

\begin{thebibliography}{10}

\bibitem{al}
C.~Arhancet and C.~Le~Merdy.
\newblock Dilation of {R}itt operators on ${L}^p$-spaces.
\newblock 2011.
\newblock Available as http://arxiv.org/ abs/1106.1513.

\bibitem{Auscher2007}
P.~Auscher.
\newblock On necessary and sufficient conditions for ${L}^p$-estimates of
  {R}iesz transforms associated to elliptic operators on ${R}^n$ and related
  estimates.
\newblock {\em Mem. Amer. Math. Soc.}, 186(871):75 pp, 2007.

\bibitem{ACDH}
P.~Auscher, T.~Coulhon, X.~T. Duong, and S.~Hofmann.
\newblock Riesz transform on manifolds and heat kernel regularity.
\newblock {\em Ann. Sci. \'Ecole Norm. Sup.}, 37(6):911--957, 2004.

\bibitem{BMartell}
N.~Badr and J.~M. Martell.
\newblock Weighted norm inequalities on graphs.
\newblock {\em J. Geom. Anal.}, 22(4):1173--1210, 2012.

\bibitem{BRuss}
N.~Badr and E.~Russ.
\newblock Interpolation of {S}obolev spaces, {L}ittlewood-{P}aley inequalities
  and {R}iesz transforms on graphs.
\newblock {\em Publ. Mat.}, 53:273--328, 2009.

\bibitem{christ}
M.~Christ.
\newblock Temporal regularity for random walk on discrete nilpotent groups.
\newblock In {\em Proceedings of the {C}onference in {H}onor of {J}ean-{P}ierre
  {K}ahane ({O}rsay, 1993)}, number Special Issue, pages 141--151, 1995.

\bibitem{CRW}
R.~R. Coifman, R.~Rochberg, and G.~Weiss.
\newblock Applications of transference: the {$L^{p}$} version of von
  {N}eumann's inequality and the {L}ittlewood-{P}aley-{S}tein theory.
\newblock In {\em Linear spaces and approximation ({P}roc. {C}onf., {M}ath.
  {R}es. {I}nst., {O}berwolfach, 1977)}, pages 53--67. Internat. Ser. Numer.
  Math., Vol. 40. Birkh\"auser, Basel, 1978.

\bibitem{CDgeq2}
T.~Coulhon and X.~T. Duong.
\newblock Riesz transform and related inequalities on noncompact {R}iemannian
  manifolds.
\newblock {\em Comm. Pure Appl. Math.}, 56(12):1728--1751, 2003.

\bibitem{CDL}
T.~Coulhon, X.~T. Duong, and X.~D. Li.
\newblock Littlewood-{P}aley-{S}tein functions on complete {R}iemannian
  manifolds for $1\leq p \leq 2$.
\newblock {\em Studia Math.}, 154(1):37--57, 2003.

\bibitem{CoulGrigor}
T.~Coulhon and A.~Grigor'yan.
\newblock Random walks on graphs with regular volume growth.
\newblock {\em Geom. Funct. Anal.}, 8(4):656--701, 1998.

\bibitem{CGZ}
T.~Coulhon, A.~Grigor'yan, and F.~Zucca.
\newblock The discrete integral maximum principle and its applications.
\newblock {\em Tohoku Math. J.}, 57:559--587, 2005.

\bibitem{CSC}
T.~Coulhon and L.~Saloff-Coste.
\newblock Puissances d'un op\'erateur r\'egularisant.
\newblock {\em Ann. Inst. H. Poincar\'e Probab. Statist.}, 26(3):419--436,
  1990.

\bibitem{Delmotte1}
T.~Delmotte.
\newblock Parabolic {H}arnack inequality and estimates of {M}arkov chains on
  graphs.
\newblock {\em Revista Matem\`atica Iberoamericana}, 15(1):181--232, 1999.

\bibitem{Dungey}
N.~Dungey.
\newblock A note on time regularity for discrete time heat kernels.
\newblock {\em Semigroups forum}, 72(3):404--410, 2006.

\bibitem{Dungey2}
N.~Dungey.
\newblock A {L}ittlewood-{P}aley-{S}tein estimate on graphs and groups.
\newblock {\em Studia Mathematica}, 189(2):113--129, 2008.

\bibitem{FS1}
C.~Fefferman and E.~M. Stein.
\newblock Some maximal inequalities.
\newblock {\em Amer. J. Math.}, 93:107--115, 1971.

\bibitem{HK}
P.~Hajlasz and P.~Koskela.
\newblock Sobolev met {P}oincar\'e.
\newblock {\em Mem. Amer. Math. Soc.}, 145(688), 2000.

\bibitem{Russ}
E.~Russ.
\newblock Riesz tranforms on graphs for $1\leq p\leq 2$.
\newblock {\em Math. Scand.}, 87(1):133--160, 2000.

\bibitem{Stein2}
E.~M. Stein.
\newblock On the {M}aximal {E}rgodic {T}heorem.
\newblock {\em Proc. Nat. Acad. Sci.}, 47:1894--1897, 1961.

\bibitem{Steinsingular}
E.~M. Stein.
\newblock {\em Singular integrals and differentiability properties of
  functions}.
\newblock Princeton Mathematical Series, No. 30. Princeton University Press,
  Princeton, N.J., 1970.

\bibitem{Stein1}
E.~M. Stein.
\newblock {\em Topics in harmonic analysis related to the {L}ittlewood-{P}aley
  theory}.
\newblock Ann. of Math. Studies. Princeton University Press, Princeton, NJ,
  1970.

\end{thebibliography}

\end{document}